\documentclass[a4paper,10pt,reqno]{article}

\usepackage[utf8]{inputenc}
\usepackage[T1]{fontenc}
\usepackage{lmodern}
\usepackage[english]{babel}
\usepackage{microtype}
\usepackage{mathtools}

\usepackage{amsmath,amssymb,amsfonts,amsthm}
\usepackage{mathtools,accents}
\usepackage[normalem]{ulem}
\usepackage{mathrsfs}
\usepackage{comment}
\usepackage{aliascnt}
\usepackage{braket}
\usepackage{bm}

\usepackage[a4paper,margin=3cm]{geometry}
\usepackage[citecolor=blue,colorlinks]{hyperref}
\usepackage{esint}
\usepackage{enumerate}
\usepackage{xcolor}

\makeatletter
\g@addto@macro\@floatboxreset\centering
\makeatother

\makeatletter
\def\newaliasedtheorem#1[#2]#3{
	\newaliascnt{#1@alt}{#2}
	\newtheorem{#1}[#1@alt]{#3}
	\expandafter\newcommand\csname #1@altname\endcsname{#3}
}
\makeatother

\numberwithin{equation}{section}

\newtheoremstyle{slanted}{\topsep}{\topsep}{\slshape}{}{\bfseries}{.}{.5em}{}

\theoremstyle{plain}
\newtheorem{theorem}{Theorem}[section]
\newaliasedtheorem{proposition}[theorem]{Proposition}
\newaliasedtheorem{lemma}[theorem]{Lemma}
\newaliasedtheorem{corollary}[theorem]{Corollary}
\newaliasedtheorem{counterexample}[theorem]{Counterexample}

\theoremstyle{definition}
\newaliasedtheorem{definition}[theorem]{Definition}
\newaliasedtheorem{question}[theorem]{Question}
\newaliasedtheorem{conjecture}[theorem]{Conjecture}

\theoremstyle{remark}
\newaliasedtheorem{remark}[theorem]{Remark}
\newaliasedtheorem{example}[theorem]{Example}

\newcommand{\setR}{\mathbb{R}}

\newcommand{\QQ}{\mathfrak Q}
\newcommand{\PP}{\mathsf{P}}

\let\altphi\phi
\let\phi\varphi
\let\varphi\altphi
\let\altphi\undefined

\newcommand{\abs}[1]{\left\lvert#1\right\rvert}
\newcommand{\norm}[1]{\left\lVert#1\right\rVert}

\DeclareMathOperator{\diam}{diam}

\newcommand{\di}{\mathop{}\!\mathrm{d}}

\newcommand{\loc}{{\rm loc}}

\newcommand{\res}{\mathop{\hbox{\vrule height 7pt width .5pt depth 0pt
			\vrule height .5pt width 6pt depth 0pt}}\nolimits}

\DeclareMathOperator{\supp}{supp}

\DeclareMathOperator{\Lip}{Lip}

\newcommand{\Geo}{{\rm Geo}}
\newcommand{\Opt}{\mathrm{OptGeo}}
\newcommand{\R}{\mathbb{R}}
\newcommand{\Ric}{\mathrm{Ric}}
\newcommand{\N}{\mathbb{N}}
\newcommand{\qq}{\mathfrak q}
\newcommand{\f}{\varphi}
\newcommand{\cI}{\mathcal{I}}

\newcommand{\dom}{{\rm Dom\,}}

\renewcommand{\L}{\mathcal{L}}

\newcommand{\dist}{\mathsf{d}}
\newcommand{\sfd}{\mathsf{d}}

\newcommand{\meas}{\mathfrak{m}}
\newcommand{\mm}{\mathfrak{m}}

\DeclareMathOperator{\CD}{CD}

\DeclareMathOperator{\RCD}{RCD}
\newfont{\tmpf}{cmsy10 scaled 2500}

\def\XXint#1#2#3{{\setbox0=\hbox{$#1{#2#3}{\int}$ }
		\vcenter{\hbox{$#2#3$ }}\kern-.6\wd0}}

\begin{document}
	
\title{Quantitative Obata's Theorem}

\author{Fabio Cavalletti\thanks{F. Cavalletti: SISSA, Trieste, Italy, email: cavallet@sissa.it} , \;
Andrea Mondino\thanks{A. Mondino: Mathematical Institut, University of Oxford, UK, email: Andrea.Mondino@maths.ox.ac.uk}  \;
 and \, Daniele Semola\thanks{D. Semola: Scuola Normale Superiore, Pisa, Italy,  email: daniele.semola@sns.it}}  %


\maketitle

\begin{abstract}
We prove a quantitative version of Obata's Theorem involving the shape of functions with null mean value when compared  with the cosine of distance functions from single points.
The deficit between the diameters of the manifold and of the corresponding sphere is bounded likewise. These results are  obtained in the general framework of (possibly non-smooth) metric measure spaces
with curvature-dimension conditions through a quantitative analysis of the transport-rays decompositions obtained by the localization method.
\end{abstract}

\tableofcontents

\section{Introduction}
One of the core topics in geometric analysis is the  deep connection between the geometry of a domain (in a possibly curved space) and spectral properties of the Laplacian defined on it.
\\The present paper focuses on the first eigenvalue $\lambda_{1}$ of the Laplacian (with Neumann boundary conditions, in case the domain has non-empty boundary). Since the Poincar\'e(-Wirtinger) inequality plays an important role in analysis and since a lower bound of the first eigenvalue gives an upper bound of the constant in the Poincar\'e(-Wirtinger) inequality, it is extremely useful to have a good lower estimate of $\lambda_{1}$. 

  For domains in the Euclidean space, classical estimates of the first eigenvalue of the Laplacian (under Dirichlet or Neumann boundary conditions) date back to Lord Rayleigh \cite{Ray},  Faber \cite{Fa23}, Krahn \cite{Kr25}, Polya-Szego \cite{PS51}, Payne-Weinberger \cite{PW}, among others. For curved spaces, two major results are due to Lichnerowicz \cite{Lichn} and Obata \cite{Obata}:
  
 \begin{theorem}\label{thm:LichnObata}
 Let $(M,g)$ be an $N$-dimensional Riemannian manifold with $\Ric_{g}\geq (N-1) g$. 
 \\Then $\lambda_{1}\geq N$ (Lichnerowicz spectral gap \cite{Lichn}). 
 \\Moreover,  $\lambda_{1}= N$ if and only if $(M,g)$ is  isometric to the unit sphere $\mathbb S^{N}$ (Obata's Theorem \cite{Obata}).
 \end{theorem} 
  
 \begin{remark}\label{rem:EigenfunctSN}
 On  $\mathbb S^{N}$, the first eigenvalue $\lambda_{1}=N$ has multiplicity $N+1$. The corresponding eigenspace is spanned by the restriction to $\mathbb S^{N}$ of affine functions of $\mathbb R^{N+1}$ (i.e. an $L^{2}$-orthogonal basis is composed by  the  standard coordinate functions $\{x^{1}, x^{2}, \ldots, x^{N+1}\}$ of $\mathbb R^{N+1}$). Equivalently, a function $u: {\mathbb S}^{N}\to \R$ is a first eigenfunction normalized as $\|u\|_{L^{2}({\mathbb S^{N}})}=1$ if and only if there exists $P\in  {\mathbb S}^{N}$  such that $u= \sqrt{N+1} \, \cos \sfd_{P}$, where we denoted by $\sfd_P$ the Riemannian distance from the point $P$.
 \end{remark} 
 
 Our main result is a \emph{quantitative spectral gap} involving the shape of the eigenfunctions (or, more generally, of functions with almost optimal Rayleigh quotient), when compared with the  eigenfunctions of the model space  ${\mathbb S}^{N}$ (as in \autoref{rem:EigenfunctSN}). In detail, we show that if $\Ric_{g}\geq (N-1) g$ and  $u:M\to \R$ is a first eigenfunction with $\|u\|_{L^{2}(M)}=1$, then there exists $P\in M$ such that
 \begin{equation}\label{eq:ucosW12}
 \|u- \sqrt{N+1} \cos\sfd_{P}\|_{L^{2}(M)} \leq C(N) (\lambda_{1}-N)^{{\rm O}(1/N)}.
 \end{equation}
 More generally, the same conclusion holds for every Lipschitz function $u:M\to \R$ with null mean value and $\|u\|_{L^{2}(M)}=1$, provided $\lambda_{1}$ on the right-hand-side is replaced by the Dirichlet energy $\int_{M} |\nabla u|^{2} d{\rm vol}_{g}$.
 \\

We will prove \eqref{eq:ucosW12} with tools of optimal transport tailored to study (possibly non-smooth) metric measure spaces satisfying Ricci curvature lower bounds and dimensional upper bounds in synthetic sense, the so-called $\CD(K,N)$ spaces introduced by   Sturm \cite{sturm:I, sturm:II} and Lott-Villani  \cite{lottvillani:metric}. For the sake of this introduction, a metric measure space (m.m.s. for short) is a triple $(X,\sfd, \mm)$ where $(X,\sfd)$ is a compact metric space and $\mm$ is a Borel probability measure, playing the role of reference volume measure. A $\CD(K,N)$ space should be roughly thought of as a possibly non-smooth metric measure space having Ricci curvature bounded below by $K\in \R$ and dimension bounded above by $N\in (1,\infty)$ in synthetic sense. The basic idea of Lott-Sturm-Villani synthetic approach is to analyse weighted convexity properties of suitable entropy functionals along geodesics in the space of probability measures endowed with the quadratic transportation (also known as Kantorovich-Wasserstein) distance. An important technical assumption throughout the paper is the essentially non-branching (``e.n.b.'' for short) property \cite{RS2014}, which roughly corresponds to requiring that the $L^{2}$-optimal transport between two absolutely continuous (with respect to the reference volume measure $\mm$) probability measures is performed along geodesics which do not branch (for the precise definitions see \autoref{s:W2}  and \autoref{Ss:geom}). 
Notable examples of spaces satisfying e.n.b. $\CD(K,N)$  include  (geodesically convex domains in) smooth Riemannian manifolds with Ricci bounded below by $K$ and dimension bounded above by $N$, their measured-Gromov-Hausdorff limits (i.e. the so-called ``Ricci limits'') and more generally $\RCD(K,N)$ spaces (i.e. $\CD(K,N)$ spaces with linear Laplacian, see \autoref{rem:RCD} for more details), finite dimensional Alexandrov spaces with curvature bounded below, Finsler manifolds endowed with a strongly convex norm. A standard example of a space failing to satisfy the essentially non-branching property is $\R^{2}$ endowed with the $L^{\infty}$ norm.   Later in the introduction, when discussing the main steps of the proof, we will mention how the essentially non-branching assumption is used in our arguments.
\\

We will establish  our results directly on the more general class of e.n.b. $\CD(N-1,N)$ metric measure spaces.  For a m.m.s. $(X,\sfd,\mm)$ we define the non-negative real number $\lambda^{1,2}_{(X,\sfd,\mm)}$ as follows
\begin{equation}\label{eq:defLambda12}
\lambda_{(X,\sfd,\mm)}^{1,2} 
: = \inf \left\{ \frac{\int_{X} |\nabla u |^{2} \, \mm}{\int_{X} |u|^{2}\,\mm} \colon u \in \Lip(X) \cap L^{2}(X,\mm), \ u \neq 0, \ \int_{X} u \, \mm  = 0\right\},
\end{equation}
where $|\nabla u|$ is the slope (also called local Lipshitz constant) of the Lipschitz function $u$ given by 
$$|\nabla u|(x)=\limsup_{y\to x} \frac{|u(x)-u(y)|}{\sfd(x,y)} \quad  \text{if $x$ is not isolated, $0$ otherwise}. $$
It is well known that, in case $(X,\sfd,\mm)$ is the m.m.s. corresponding to a smooth compact Riemannian manifold (possibly with boundary), then  
$\lambda^{1,2}_{(X,\sfd,\mm)}$ coincides with the first  eigenvalue of the problem $-\Delta u= \lambda u$ with Neumann boundary conditions.
\\

Considering the extension of   \eqref{eq:ucosW12} to e.n.b. $\CD(N-1,N)$ spaces  is natural: indeed a sequence $(M_{j}, g_{j})$ of Riemannian $N$-manifolds with $\Ric_{g_{j}}\geq (N-1) g_{j}$ where the right hand side of  \eqref{eq:ucosW12} converges to zero as $j\to \infty$  may develop singularities and admits a limit (up to subsequences) in the measured-Gromov-Hausdorff sense to a possibly non-smooth e.n.b. $\CD(N-1,N)$ space (actually the limit is, more strongly, $\RCD(N-1,N)$).
\\

In the enlarged class of e.n.b. $\CD(N-1,N)$ spaces (actually already for  $\RCD(N-1,N)$ spaces), Obata's rigidity Theorem must be modified: 
\begin{itemize}
\item First of all, $N\in (1,\infty)$ is a (possibly non integer) real number; 
\item Even in the case of integer $N$, the round sphere $\mathbb S^{N}$ is not anymore the only case of equality in the Lichnerowicz spectral gap as the spherical suspensions achieve equality as well \cite{Ketterer15}. 
\end{itemize}
A key geometric property of the spherical suspensions is that they have diameter $\pi$, thus saturating Bonnet-Myers diameter upper bound. The first part of our main result is a quantitative control of how close to $\pi$ the diameter must be,  in terms of the spectral gap deficit. The second part of the statement is an $L^{2}$-quantitative control of the shape of functions with almost optimal Rayleigh quotient. We can now state our main theorem.

\begin{theorem} [Quantitative Obata's Theorem for e.n.b. $\CD(N-1,N)$-spaces]\label{thm:quantobataIntro}
For every real number  $N>1$ there exists  a real constant  $C(N)>0$  with the following properties:
if $(X,\dist,\meas)$ is an essentially non branching metric measure space satisfying the  $\CD(N-1,N)$ condition and $\mm(X)=1$ with $\supp(\mm)=X$, then
\begin{equation}\label{eq:diamXpiIntro}
\pi-\diam(X) \leq C(N) \big(\lambda_{(X,\sfd,\mm)}^{1,2} -N \big)^{1/N}.
\end{equation}
Moreover, for any Lipschitz function $u:X\to \R$ with $\int_X u \, \mm=0$ and $\int_X u^2 \,\mm=1$, there exists a distinguished point $P\in X$ such that
\begin{equation}\label{eq:ucosdpL2}
\begin{split}
\norm{u-\sqrt{N+1}\cos \dist_{P}}_{L^2(X,\meas)} &\le C(N) \left(\int_X\abs{\nabla u}^2\mm-N\right)^{\eta}  \\
\eta&=\frac{1}{6N+4}
\end{split}.
\end{equation}
\end{theorem}
 
 \begin{remark}
Although \autoref{thm:quantobataIntro} is formulated for e.n.b. $\CD(N-1,N)$ spaces, a statement for  e.n.b. $\CD(K,N)$ spaces with $K>0$ is easily obtained by scaling. Indeed, $(X,\sfd,\mm)$ satisfies $\CD(K,N)$ if and only if, for any $\alpha, \beta\in (0,\infty)$, the scaled metric measure space $(X,\alpha \sfd, \beta \mm)$ satisfies $\CD(\alpha^{-2} K, N)$; see \cite[Proposition 1.4]{sturm:II}.
 \end{remark}
 
Let us compare \autoref{thm:quantobataIntro} with related results in the literature. Under the standing assumption that $(M,g)$ is a \emph{smooth} Riemannian $N$-manifold \emph{without boundary} and with $\Ric_{g}\geq (N-1)g$:
 \begin{enumerate}
\item  It follows from Cheng's Comparison Theorem \cite{Cheng75} that if $\lambda^{1,2}_{(M,g)}$ is close to $N$ then the diameter of $M$ must be close to $\pi$. Conversely, Croke \cite{Croke82} proved that if the diameter is close to $\pi$ then $\lambda^{1,2}_{(M,g)}$ must be close to $N$. B\'erard-Besson-Gallot \cite{BerardBessonGallot85}  sharpened the diameter estimate of Cheng by proving an estimate very similar to \eqref{eq:diamXpiIntro}. 
\item Bertrand \cite{bertrand07} established the following stability result for eigenfunctions (see also the prior work of Petersen \cite{Pet99}): for every $\varepsilon>0$ there exists $\delta>0$ such that if $\lambda_{1}\leq N+\delta$ and $u$ is an eigenfunction relative to $\lambda_{1}$ normalized so that $\int_{M} u^{2} d{\rm vol}_{g}= {\rm vol}_{g}(M)$, then there exists a point $P\in M$ such that $\|u-\sqrt{N+1}\cos \sfd_{P} \|_{L^{\infty}(X,\mm)} \leq \varepsilon$. 
\end{enumerate}
\autoref{thm:quantobataIntro} sharpens and extends the above results in various ways:
\begin{itemize}
\item The estimate \eqref{eq:diamXpiIntro} extends  \cite{BerardBessonGallot85}  to e.n.b. $\CD(N-1,N)$ spaces. These spaces are non-smooth a priori and may have (convex) boundary.
 Actually, as the reader will realize, the claim \eqref{eq:diamXpiIntro} will be set in \autoref{S:Obatadiam}
along the way of proving the much harder \eqref{eq:ucosdpL2}, to which the entire \autoref{sec:QuantObataFuct} is devoted.
\item The estimate \eqref{eq:ucosdpL2} extends Bertrand's \cite{bertrand07} stability to the more general  class of e.n.b. $\CD(N-1,N)$ spaces and to arbitrary functions (a priori not eigenfunctions) with Rayleigh quotient close to $N$. The fact that $u$ is an eigenfunction was key in  \cite{bertrand07} in order to apply maximum principle and gradient estimates in the spirit of Li-Yau \cite{LiYau}. Let us stress that our methods are completely different and work for an arbitrary Lipschitz function satisfying a \emph{small energy} condition but \emph{no PDE} a-priori.
\end{itemize}

Inequality  \eqref{eq:ucosdpL2} naturally fits in the framework of quantitative functional/geometric inequalities. A basic result in this context is the quantitative euclidean  isoperimetric inequality  proved by Fusco-Maggi-Pratelli \cite{FuMaPr} (see also  \cite{FiMaPr,CiLe} for different proofs),  stating that for every Borel set $E\subset\R^n$ of positive and finite volume  there exists $\bar{x}\in\R^n$ such that
\begin{equation}
  \label{eq:EII}
  \frac{|E\Delta B_{r_E}(\bar{x})|}{|E|}\le C(N)\,\Big(\frac{\PP(E)}{\PP(B_{r_E}(\bar{x}))}-1\Big)^{1/2}
\end{equation}
where $r_E$ is such that $|B_{r_E}(\bar{x})|=|E|$. Quantitative results involving the spectrum of the Laplacian have been proved for domains in $\R^{n}$,  among others, by Hansen-Nadirashvili \cite{HaNa} in dimension 2,  Melas \cite{Mel} for convex bodies, Fusco-Maggi-Pratelli \cite{FuMaPr2}, Brasco-De Philippis-Velichkov \cite{BrDepVel} regarding quantitative forms of the Faber-Krahn inequality and by Brasco-Pratelli \cite{BrPr} regarding quantitative versions of the Krahn-Szego and Szego-Weinberger inequalities. More recently, a quantitative version of the L\'evy-Gromov isoperimetric inequality has been proved for essentially non branching $\CD(N-1,N)$ metric measure spaces in \cite{CMM}, and a quantitative isoperimetric inequality in the setting of smooth Riemannian manifolds has been considered in \cite{Chodosh19}.

Taking variations in the broad context of metric measure spaces makes the prediction on the sharp exponent $\eta$ in \eqref{eq:ucosdpL2} a hard task. Even formulating a conjecture is  a challenging question and 
it could actually be that $\eta={\rm O}(1/N)$ as $N\to\infty$ is already sharp. In the direction of this guess, we notice that the exponent $1/N$ in \eqref{eq:diamXpiIntro} is indeed optimal in the class of metric measure spaces, as a direct computation on the model 1-dimensional space $([0,D], |\cdot|, \sin^{N-1}(\cdot) \L^{1})$ shows.

Before discussing the main steps in the proof of \autoref{thm:quantobataIntro}, it  is worth recalling remarkable examples of spaces fitting in the assumptions of the result. Let us stress that our main theorem seems new  in all of them.
The class of essentially non branching  $\CD(N-1,N)$ spaces includes many notable families of spaces, among them:
\begin{itemize}
\item \emph{Geodesically convex domains} in (resp. weighted) Riemannian $N$-manifolds satisfying ${\Ric}_g\ge (N-1)g$   (resp. $N$-Bakry-\'Emery Ricci curvature bounded below by $N-1$).  
\item \emph{Measured Gromov Hausdorff limits of Riemannian $N$-manifolds  satisfying  ${\Ric}_g\ge (N-1)g$} (so called ``Ricci limits'') and more generally the class of $\RCD(N-1,N)$ spaces.
Indeed Ricci limits  are examples of $\RCD(N-1,N)$ spaces (see for instance \cite{GMS2013}) and $\RCD(N-1,N)$ spaces are essentially non-branching $\CD(N-1,N)$ (see \cite{RS2014}).
\item \emph{Alexandrov spaces with curvature $\geq 1$}.
Petrunin \cite{PLSV} proved that the synthetic curvature lower bound in the sense of comparison triangles is compatible with the optimal transport lower bound on the Ricci curvature of Lott-Sturm-Villani (see also \cite{zhangzhu}).  Moreover  geodesics in Alexandrov spaces with curvature bounded below do not branch. It follows that  Alexandrov spaces with curvature bounded from below by $1$ and Hausdorff dimension at most $N$ are non-branching  $\CD(N-1,N)$ spaces.
\item \emph{Finsler manifolds with strongly convex norm, and satisfying Ricci curvature lower bounds.} More precisely we consider a $C^{\infty}$-manifold  $M$, endowed with a function $F:TM\to[0,\infty]$ such that $F|_{TM\setminus \{0\}}$ is $C^{\infty}$ and  for each $x \in M$ it holds that $F_x:=T_{x}M\to [0,\infty]$ is a  strongly-convex norm, i.e.
$$\qquad \quad \; g^x_{ij}(v):=\frac{\partial^2 (F_x^2)}{\partial v^i \partial v^j}(v) \quad \text{is a positive definite matrix at every } v \in T_xM\setminus\{0\}. $$
Under these conditions, it is known that one can write the  geodesic equations and geodesics do not branch: in other words these spaces are non-branching.
We also assume $(M,F)$ to be geodesically complete and endowed with a $C^{\infty}$ probability measure $\mm$ in  such a way that the associated m.m.s. $(X,F,\mm)$ satisfies the $\CD(N-1,N)$ condition. This class of spaces has been investigated by Ohta \cite{Ohta} who established the equivalence between the Curvature Dimension condition and a Finsler-version of Bakry-Emery $N$-Ricci tensor bounded from below.
\end{itemize}

\subsection*{An overview of the proof}
The starting point of the proof of \autoref{thm:quantobataIntro} is the metric measured version of the classical {\it localization technique}. First introduced by Payne-Weinberger \cite{PW} for establishing a sharp Poincar\'e-Wirtinger inequality for convex domains in $\R^{n}$, the localization technique has been developed into a general dimension reduction tool for geometric inequalities in symmetric spaces by Gromov-Milman \cite{GrMi}, Lov\'asz-Simonovits \cite{LoSi} and Kannan-Lov\'asz-Simonovits \cite{KaLoSi}. More recently, Klartag  \cite{klartag} used optimal transportation tools in order to extend the range of applicability of the techique to general Riemannian manifolds. The extension to the metric setting was finally obtained in \cite{CavallettiMondino17a}, see \autoref{Ss:localization}.

Given a function $u\in L^{1}(X,\mm)$ with $\int_{X} u \, \mm=0$, the localization theorem (\autoref{T:localize}) gives a decomposition of $X$ into a family of one-dimensional sets  $\{X_{q}\}_{q \in Q}$ formed by the transport rays of a Kantorovich potential associated to the optimal transport from the positive part of $u$ (i.e. $\mu_{0}:=\max\{u,0\} \, \mm$) to the negative part of $u$  (i.e. $\mu_{1}:=\max\{-u,0\}\, \mm$); each $X_{q}$ 
is in particular isometric to a real interval.  A first key property of such a decomposition is that each ray $X_{q}$ carries a  natural measure $\mm_{q}$ (given by the the Disintegration Theorem) in such a way that
\begin{equation}\label{eq:localizedGeomConstr}
 \text{$(X_{q}, \sfd, \mm_{q})$ is a $\CD(N-1,N)$ space and $\int_{X_{q}} u \, \mm_{q}=0$}, 
 \end{equation}
 so that both the geometry of the space and the null mean value constraint  are \emph{localized} into a family of one-dimensional spaces. 
An important ingredient used in the proof of such a decomposition is the essentially non-branching property  which, coupled with $\CD(N-1,N)$ (actually the weaker measure contraction would suffice here), guarantees that the rays form a partition of $X$ (up to an $\mm$-negligible set).   

In order to exploit \eqref{eq:localizedGeomConstr}, as a first step, in \autoref{S:1d} we prove the one dimensional counterparts of \autoref{thm:quantobataIntro}. More precisely, given a  $1$-dimensional $\CD(N-1,N)$ space $(I=[0,D], |\cdot|, \mm)$ we show that (\autoref{thm:improvedspectralNeumann1D})
\begin{equation}\label{eq:almMaxDiam1D}
\pi-D\le  C(N) (\lambda^{1,2}_{(I,|\cdot|,\meas)}-N)^{1/N}, 
\end{equation}
 and that, if  $u\in\Lip(I)$ satisfies $\int u \,\mm=0$ and $\int u^2 \,\mm=1$, then  (\autoref{thm:mainthm1d})
\begin{equation}\label{eq:thm1dIntro}
\min\left\lbrace\norm{u-\sqrt{N+1}\cos(\cdot)}_{L^{2}(\meas)},\norm{u+\sqrt{N+1}\cos(\cdot)}_{L^{2}(\meas)} \right\rbrace\le C\left(\int \abs{u'}^2\mm-N  \right)^{\min\left\lbrace \frac{1}{2}, \frac{1}{N}\right\rbrace }.
\end{equation}
Combining \eqref{eq:localizedGeomConstr} and  \eqref{eq:almMaxDiam1D} it is not hard to prove (see  \autoref{thm:improvedspectralNeumann}) the first claim \eqref{eq:diamXpiIntro} of \autoref{thm:quantobataIntro}. Actually,  calling $Q_{\ell}$  (for ``$Q$ long'') the set of indeces for which  $|X_{q}|\simeq \pi$, we aim to show that $\qq(Q_{\ell})\simeq 1$  (i.e. ``most rays are long'').  As we will discuss in a few lines, this is far from being trivial (in particular, it needs new ideas when compared with \cite{CMM}).

 A second crucial property of the decomposition $\{X_{q}\}_{q \in Q}$, inherited by the variational nature of the construction, is the so-called cyclical monotonicity. This was key in \cite{CMM} for showing that, for $q\in Q_{\ell}$,  the transport ray $X_{q}$ has its starting point close to a fixed ``south pole'' $P_{S}$, and ends-up nearby a fixed ``north pole'' $P_{N}$ (in particular, the distance between $P_{S}$ and $P_{N}$ is close to $\pi$) (\autoref{prop:PNPS}).

 Then  we  observe that \eqref{eq:thm1dIntro}  forces, for  $q\in Q_{\ell}$, the fiber  $u_{q}:=u\llcorner X_{q}$ (that is the restriction of $u$ to the corresponding one dimensional  element of the partition) to be  $L^{2}$ close to a multiple of the cosine of the arclength parametrization along the ray $X_{q}$, i.e. 
\begin{equation} \label{eq:uqcqcos}
u_{q}(\cdot)  \simeq c_{q} \sqrt{N+1}\cos(\cdot)\; \text{along $X_{q}$, where  } c_{q}=\|u_{q}\|_{L^{2}(\mm_{q}),}\;  \text{ for } q\in Q_{\ell} \;\text{ (see \eqref{eq:1Duqcosdq}).}
\end{equation}
The difficulties in order to conclude the proof are mainly two, and are strictly linked:
\begin{enumerate}
\item  Show that $Q_{\ell}\ni q\mapsto c_{q}$ is almost constant;
\item Show that   $\qq(Q_{\ell}) \simeq 1$.
\end{enumerate}
Let us stress that at this stage the only given information is that $\int_{Q_{\ell}} c^{2}_{q}\, \qq \simeq 1$.  The intuition why 1. and 2. should hold  is that an oscillation of $c_{q}$ would correspond to an oscillation of $u$ ``orthogonal to the transport rays'', which would be expensive in terms of Dirichlet energy of $u$. The proofs of the two claims are the most technical part of the work and correspond respectively to  \autoref{prop:varcqlongrays} and \autoref{prop:massoflongraysbound}.
\\Let us mention that the two difficulties 1. and 2. were not present in the proof of the quantitative L\'evy-Gromov inequality in \cite{CMM}, where it was sufficient to work with characteristic functions (which have a fixed scale, i.e. they are either $0$ or $1$).

\medskip

\noindent{\bf Acknowledgement:} 
Part of the work was developed when D.S. was visiting F.C. at SISSA and  A.M. at the Mathematics Institute of the University of Warwick; he wishes to thank both institutes for the excellent working conditions and the stimulating atmosphere.
\\ A.M. is  supported by the EPSRC First Grant EP/R004730/1 ``Optimal transport and geometric analysis'' and by the ERC Starting Grant  802689 ``CURVATURE''.\\
The authors are grateful to the anonymous reviewers, for their suggestions that helped to improve a previous version of the paper.

\section{Background material}\label{S:back}
The goal of this section is to fix the notation and to recall the basic notions/constructions used throughout the paper: in \autoref{s:W2} we review geodesics in the Wasserstein distance, in \autoref{Ss:geom} curvature-dimension conditions, in \autoref{Ss:1DimCD(K,N)} some basics of $\CD(K,N)$ densities on segments of the Real line, and in \autoref{Ss:localization}  the decomposition of the space into transport rays (localization).


\subsection{Geodesics in the $L^2$-Kantorovich-Wasserstein distance}
\label{s:W2} 

Let $(X,\sfd)$ be a compact metric space and $\mm$ a Borel probability measure over $X$. The triple $(X,\sfd, \mm)$ is called metric measure space, m.m.s. for short.
\\ The space of all Borel probability measures over $X$ will be denoted by $\mathcal{P}(X)$.
We define the $L^{2}$-Kantorovich-Wasserstein distance $W_{2}$ between two measures  $\mu_0,\mu_1 \in \mathcal{P}(X)$ as
\begin{equation}\label{eq:Wdef}
  W_2(\mu_0,\mu_1)^2 = \inf_{\pi} \int_{X\times X} \sfd^2(x,y) \, \pi(dxdy),
\end{equation}
where the infimum is taken over all $\pi \in \mathcal{P}(X \times X)$ with $\mu_0$ and $\mu_1$ as the first and the second marginal, i.e. $(P_{1})_{\sharp} \pi= \mu_{0},  (P_{2})_{\sharp} \pi= \mu_{1}$. Of course $P_{i}, i=1,2$ denotes the projection on the first (resp. second) factor and $(P_{i})_{\sharp}$ is the corresponding push-forward map on measures.  As $(X,\sfd)$ is complete, also $(\mathcal{P}(X), W_{2})$ is complete.

The space of geodesics of $(X,\sfd)$ is denoted by
$$
\Geo(X) : = \big\{ \gamma \in C([0,1], X):  \sfd(\gamma_{s},\gamma_{t}) = |s-t| \sfd(\gamma_{0},\gamma_{1}), \text{ for every } s,t \in [0,1] \big\}.
$$
A metric space $(X,\sfd)$ is said to be a \emph{geodesic space} if and only if for each $x,y \in X$ there exists $\gamma \in \Geo(X)$ such that $\gamma_{0} =x, \gamma_{1} = y$. A basic fact of $W_{2}$ geometry, is that if $(X,\sfd)$ is geodesic then $(\mathcal{P}(X), W_2)$ is geodesic as well.
For any $t\in [0,1]$, let ${\rm e}_{t}$ denote the evaluation map:
$$
  {\rm e}_{t} : \Geo(X) \to X, \qquad {\rm e}_{t}(\gamma) : = \gamma_{t}.
$$
Any geodesic $(\mu_t)_{t \in [0,1]}$ in $(\mathcal{P}(X), W_2)$  can be lifted to a measure $\nu \in {\mathcal {P}}(\Geo(X))$, called \emph{dynamical optimal plan},
such that $({\rm e}_t)_\sharp \, \nu = \mu_t$ for all $t \in [0,1]$.
Given $\mu_{0},\mu_{1} \in \mathcal{P}(X)$, we denote by
$\Opt(\mu_{0},\mu_{1})$ the space of all $\nu \in \mathcal{P}(\Geo(X))$ for which $({\rm e}_0,{\rm e}_1)_\sharp\, \nu$
realizes the minimum in \eqref{eq:Wdef}. Here as usual $\sharp$ indicates the push-forward operation.
If $(X,\sfd)$ is geodesic, then the set  $\Opt(\mu_{0},\mu_{1})$ is non-empty for any $\mu_0,\mu_1\in \mathcal{P}(X)$.

A set $F \subset \Geo(X)$ is a \emph{set of non-branching geodesics} if and only if for any $\gamma^{1},\gamma^{2} \in F$, it holds:
$$
\exists \;  \bar t\in (0,1) \text{ such that } \ \forall t \in [0, \bar t\,] \quad  \gamma_{ t}^{1} = \gamma_{t}^{2}
\quad
\Longrightarrow
\quad
\gamma^{1}_{s} = \gamma^{2}_{s}, \quad \forall s \in [0,1].
$$
A measure $\mu$ on a measurable space $(\Omega,\mathcal{F})$ is said to be \emph{concentrated}
on $F \subset \Omega$ if $\exists E \subset F$ with $E \in \mathcal{F}$ so that $\mu(\Omega \setminus E) = 0$. With this terminology, we next recall the definition of \emph{essentially non-branching space}  from \cite{RS2014}. 
\begin{definition}\label{D:essnonbranch}
A metric measure space $(X,\sfd, \mm)$ is \emph{essentially non-branching} if and only if for any $\mu_{0},\mu_{1} \in \mathcal{P}(X)$,
with $\mu_{0},\mu_{1}$ absolutely continuous with respect to $\mm$, any element of $\Opt(\mu_{0},\mu_{1})$ is concentrated on a set of non-branching geodesics.
\end{definition}

\subsection{Curvature-dimension conditions for metric measure spaces}\label{Ss:geom}
 The $L^2$-transport structure described in \autoref{s:W2}  allows to formulate a generalized notion of Ricci curvature lower bound coupled with a dimension upper bound in the context of possibly non-smooth metric measure spaces. This corresponds to the $\CD(K,N)$ condition introduced in the seminal works of Sturm \cite{sturm:I, sturm:II} and Lott--Villani \cite{lottvillani:metric}, which here is reviewed only for a compact m.m.s. $(X,\sfd,\mm)$ with $\mm \in \mathcal{P}(X)$ and in case $K > 0, 1 <N<\infty$ (the basic setting of the present paper).

For $N \in (1,\infty)$, the {\it $N$-R\'enyi relative-entropy functional}
$\mathcal{E}_N : \mathcal{P}(X) \rightarrow [0,1]$ is defined as
\[
\mathcal{E}_N(\mu) := \int \rho^{1 - \frac{1}{N}} d\mm \,,
\]
where $\mu = \rho \mm + \mu^{sing}$ is the Lebesgue decomposition of $\mu$ with $\mu^{sing} \perp \mm$.
\\Given $K\in(0,\infty)$, $N \in(1,\infty)$, and $t \in [0,1]$, define $\sigma^{(t)}_{K,N}:[0,\infty)\to[0,\infty]$ as follows
\begin{equation} \label{def:sigma}
\begin{cases}
\sigma^{(t)}_{K,N}(0) := t \\
\sigma^{(t)}_{K,N}(\theta) := \frac{\sin\left(t \theta \sqrt{\frac{K}{N}}\right)}{\sin\left(\theta \sqrt{\frac{K}{N}}\right)}\quad \text{ if } 0 < \theta < \frac{\pi}{\sqrt{K/N}}\\
\sigma^{(t)}_{K,N}(\theta) := +\infty \quad \text{ otherwise }.
\end{cases}
\end{equation}
Set  also
\begin{equation}\label{eq:deftau}
\tau_{K,N}^{(t)}(\theta) := t^{\frac{1}{N}} \sigma_{K,N-1}^{(t)}(\theta)^{1 - \frac{1}{N}}\,.
\end{equation}

\begin{definition}[$\CD(K,N)$] \label{def:CDKN}
A m.m.s. $(X,\sfd,\mm)$ is said to satisfy $\CD(K,N)$ if for all $\mu_0,\mu_1 \in \mathcal{P}(X)$ absolutely continuous with respect to $\mm$
there exists $\nu \in \Opt(\mu_0,\mu_1)$ so that for all $t\in[0,1]$ it holds $\mu_t := ({\rm e}_{t})_{\#} \nu \ll \mm$ and 
\begin{equation} \label{eq:CDKN-def}
\mathcal{E}_{N'}(\mu_t) \geq \int_{X \times X} \left( \tau^{(1-t)}_{K,N'}(\sfd(x_0,x_1)) \rho_0^{-1/N'}(x_0) + \tau^{(t)}_{K,N'}(\sfd(x_0,x_1)) \rho_1^{-1/N'}(x_1) \right) \pi(dx_0,dx_1) , 
\end{equation}
for all $N'\geq N$, where $\pi = ({\rm e}_0,{\rm e}_1)_{\sharp}(\nu)$ and $\mu_i = \rho_i \mm$, $i=0,1$.
\end{definition}

If $(X,\sfd,\mm)$ verifies the $\CD(K,N)$ condition then the same is valid for $(\supp(\mm),\sfd,\mm)$; hence we directly assume $X = \supp(\mm)$.

For the general definition of $\CD(K,N)$ see \cite{lottvillani:metric, sturm:I, sturm:II}.

\begin{remark}[Case of a smooth Riemannian manifold]
 It is worth recalling that if $(M,g)$ is a Riemannian manifold of dimension $n$ and
$h \in C^{2}(M)$ with $h > 0$ then, denoting with $\sfd_{g}$ and $vol_{g}$ the Riemannian distance and volume measure,  the m.m.s.  $(M,\sfd_{g},h \, vol_{g})$  verifies $\CD(K,N)$ with $N\geq n$ if and only if  (see \cite[Theorem 1.7]{sturm:II})
$$
Ric_{g,h,N} \geq  K g, \qquad Ric_{g,h,N} : =  Ric_{g} - (N-n) \frac{\nabla_{g}^{2} h^{\frac{1}{N-n}}}{h^{\frac{1}{N-n}}},
$$
 in other words if and only if the weighted Riemannian manifold $(M,g, h \, vol_{g})$ has $N$-Bakry-\'Emery Ricci tensor bounded below by $K$.
Note that if $N = n$ the  Bakry-\'Emery  Ricci tensor $Ric_{g,h,N}= Ric_{g}$ makes sense only if $h$ is constant.
\hfill$\Box$
\end{remark}
\medskip

\begin{remark}[$\CD^{*}(K,N)$, $\RCD^{*}(K,N)$ and $\RCD(K,N)$]\label{rem:RCD}
The lack of the local-to-global property of the $\CD(K,N)$ condition (for $K/N \neq 0$)
led in 2010 Bacher and Sturm to introduce in \cite{BS10} the reduced curvature-dimension condition, denoted by $\CD^{*}(K,N)$.
The $\CD^{*}(K,N)$ condition asks for the same inequality \eqref{eq:CDKN-def} of $\CD(K,N)$ to hold but
the coefficients $\tau_{K,N}^{(s)}(\sfd(\gamma_{0},\gamma_{1}))$ are replaced by the slightly smaller $\sigma_{K,N}^{(s)}(\sfd(\gamma_{0},\gamma_{1}))$. Let us explicitly notice that, in general, $\CD^{*}(K,N)$  is weaker than $\CD(K,N)$.
A subsequent breakthrough in the theory was obtained with the introduction of the Riemannian curvature dimension condition $\RCD(K,N)$:
in the infinite dimensional case $N = \infty$ this condition was introduced in \cite{AGS11b} (for finite measures $\mm$, and in \cite{AGMR12} for $\sigma$-finite ones).
The finite dimensional refinements $\RCD(K,N)/\RCD^{*}(K,N)$ with $N<\infty$ were subsequently studied in \cite{gigli:laplacian,EKS,AMS}. We refer to these articles as well as to the survey papers \cite{AmbrosioICM,VilB} for a general account
on the synthetic formulation of  Ricci curvature lower bounds, in particular of the latter Riemannian-type.
Here we only briefly recall that it is a stable \cite{GMS2013} strengthening of the (resp. reduced) curvature-dimension condition:
a m.m.s. verifies $\RCD(K,N)$ (resp.  $\RCD^{*}(K,N)$) if and only if it satisfies $\CD(K,N)$ (resp. $\CD^{*}(K,N)$) and the Sobolev space $W^{1,2}(X,\mm)$ is a Hilbert space (with the Hilbert structure induced by the Cheeger energy).

To conclude we recall also that recently, the first named author together with E. Milman  \cite{CMi} proved the equivalence
of $\CD(K,N)$ and $\CD^{*}(K,N)$,
together with the local-to-global property for $\CD(K,N)$, in the framework of  essentially non-branching m.m.s. having $\mm(X) < \infty$.
As we will always assume the aforementioned properties to be satisfied by our ambient m.m.s. $(X,\sfd,\mm)$, we will use both formulations with no distinction.
It is worth also mentioning that a m.m.s. verifying $\RCD^{*}(K,N)$ is essentially non-branching (see \cite[Corollary 1.2]{RS2014})
implying also the equivalence of $\RCD^{*}(K,N)$ and  $\RCD(K,N)$ (see \cite{CMi} for details). \hfill$\Box$
\end{remark}

\medskip

We shall always assume that the m.m.s. $(X,\sfd,\mm)$ is essentially non-branching and satisfies $\CD(K,N)$ for some $K>0, N\in (1,\infty)$
with $\supp(\mm) = X$. It follows that $(X,\sfd)$ is a geodesic and compact metric space. More precisely: note we assumed from the beginning  $(X,\sfd)$ to be compact for sake of simplicity, however such an assumption could have been replaced by completeness and separability throughout \autoref{s:W2} and \autoref{Ss:geom}; however compactness would have been now a consequence of  $\CD(K,N)$ for some $K>0, N\in (1,\infty)$.

A useful property of essentially non-branching $\CD(K,N)$ spaces is the validity of a weak local Poincar\'e inequality. 

\begin{proposition}[Weak local Poincar\'e inequality]
Let $(X,\sfd,\mm)$ be an essentially non-branching $\CD(K,N)$ space for some $K\geq 0, N>1$.
For every $u\in \Lip(X)$ it holds
\begin{equation}\label{eq:LocPoinc11}
 \fint_{B_{r}(x)}\Big|u - \fint_{B_{r}(x)} u\Big| \mm\\
\leq  2^{N+2} r \fint_{B_{2r}(x)} |\nabla u| \mm.
\end{equation}
More generally, for every $p\geq 1$ there exists $C_{p,N}$ such that
\begin{equation}\label{eq:LocPoincqq}
 \fint_{B_{r}(x)}\Big|u - \fint_{B_{r}(x)} u\Big|^{p} \mm\\
\leq  C_{p,N} \, r^p \,\fint_{B_{10r}(x)} |\nabla u|^{p} \mm.
\end{equation}
\end{proposition}

\begin{proof}
It is well known that, in essentially  non-branching $\CD(K,N)$ spaces, the $W_{2}$ geodesic connecting two absolutely continuous probability measures is unique (indeed, it holds more generally for essentially non-branching ${\rm MCP}(K,N)$ spaces, \cite[Theorem 1.1]{CavallettiMondino17c}). Thus, $(X,\sfd,\mm)$ as in the assumptions enters the framework of \cite[Corollary 1]{RajalaLocPoinc} and  \eqref{eq:LocPoinc11} follows.

Recalling that by Bishop-Gromov Inequality \cite[Theorem 2.3]{sturm:II} it holds $\frac{\mm(B_{\rho}(x_{0}))}{\mm(B_{1}(x_{0}))}\geq C_{N} \rho^{N}$ for every $\rho\in [0,1], x_{0}\in X$, the second claim \eqref{eq:LocPoincqq} is a consequence of \eqref{eq:LocPoinc11} and \cite[Theorem 5.1]{HaKo}.
\end{proof}


\subsection{$\CD(K,N)$ densities on segments of the real line}
\label{Ss:1DimCD(K,N)}
We will use several times the following terminology:
recalling the coefficients $\sigma$ from \eqref{def:sigma}, a non-negative function $h$ defined on an interval $I \subset \R$ is called \emph{a $\CD(K,N)$ density on $I$}, for $K \in \R$ and $N \in (1,\infty)$, if for all $x_0,x_1 \in I$ and $t \in [0,1]$:
\begin{equation}\label{E:1dCD}
 h(t x_1 + (1-t) x_0)^{\frac{1}{N-1}} \geq  \sigma^{(t)}_{K,N-1}(| x_1-x_0|) h(x_1)^{\frac{1}{N-1}} + \sigma^{(1-t)}_{K,N-1}(|x_1-x_0|) h(x_0)^{\frac{1}{N-1}} .
\end{equation}

The link with the definition of $\CD(K,N)$ for m.m.s. can be summarized as follows (see for instance \cite[Theorem A.2]{CMi}):
if $h$ is a $\CD(K,N)$ density on an interval $I \subset \R$ then the m.m.s. $(I,|\cdot |,h(t) dt)$ verifies $\CD(K,N)$; conversely, if the m.m.s. $(\R,|\cdot |,\mu)$
verifies $\CD(K,N)$ and $I = \supp(\mu)$ is not a point, then $\mu \ll \L^1$ and there exists a representative of the density $h = d\mu / d\L^1$ which is a $\CD(K,N)$ density on $I$.

A $\CD(K,N)$ density  $h$ defined on an interval $I \subset \R$ satisfies the following properties:
\begin{itemize}
\item $h$ is lower semi-continuous on $I$ and locally Lipschitz continuous in its interior (this is easily reduced to the corresponding statement for concave functions on $I$).
\item $h$ is strictly positive in the interior of $I$ whenever it does not vanish identically (this follows directly from the definition \eqref{E:1dCD}).
\item $h$ is locally semi-concave in the interior of $I$, i.e. for all $x_{0}$ in the interior of $I$, there exists $C_{x_{0}}\in \R$ so that $x\mapsto h(x)-C_{x_{0}} (x-x_{0})^{2}$ is concave in a neighbourhood of $x_{0}$. In particular, $h$ is twice differentiable in $I$ with at most countably many exceptions.
\end{itemize}

As proven in  \cite[Lemma A.5]{CMi}, if $h$ is a $\CD(K,N)$ density on an interval $I$ then at any point $x$ in the interior where it is twice differentiable (thus up to  at most countably many exceptions) it holds
\begin{equation}\label{eq:logh''-K}
(\log h)''(x)+\frac{1}{N-1}\left((\log h)'(x) \right)^{2} = (N-1) \frac{(h^{\frac{1}{N-1}})''  (x)}{h^{\frac{1}{N-1}}(x)}\leq -K.
\end{equation}
Also the converse implication holds, see  \cite[Lemma A.6]{CMi} for the proof and the precise statement. 

We next recall some estimates on $\CD(N-1,N)$ densities, which will turn up to be useful in the paper. 
Let $h_N$ be the model density for the $\CD(N-1,N)$ condition given by
\begin{equation}\label{eq:densityhN}
h_N(t):=\frac{1}{\omega_N}\sin^{N-1}(t)\quad\text{for $t\in[0,\pi]$},
\end{equation}
where $\omega_N:=\int_0^{\pi}\sin^{N-1}(t)\di t$. Let $\epsilon:=\pi-D$ and $\lambda_D:=\int_0^Dh_N(t)\di t$, for any $D\in [0,\pi]$.
\\For a proof of the next proposition see for instance \cite[Proposition A.3]{CMM}.
\begin{proposition}\label{prop:hCDN-1N}
Let $h:[0,D]\to[0,+\infty)$ be a $\CD(N-1,N)$ density which integrates to $1$ on $[0,D]$. Then, for any $t\in(0,D)$, it holds
\begin{equation}\label{eq:distdens1}
\left(\frac{\omega_N}{\omega_N\lambda_D+\epsilon}\right)\min\left\lbrace h_N(t),h_N(t+\epsilon)\right\rbrace \le h(t)\le \left(\frac{\omega_N}{\omega_N-\epsilon}\right)\max\left\lbrace h_N(t),h_N(t+\epsilon)\right\rbrace.
\end{equation}
\end{proposition}

\begin{corollary}\label{cor:estdens}
Under the assumptions of \autoref{prop:hCDN-1N}, there exist a constant $C=C(N)>0$ and $\epsilon_0>0$ with the following property: if $\epsilon\in [0,\epsilon_0]$ then for any $t\in (0,D)$ it holds
\begin{equation}\label{eq:distancedensities}
\abs{h(t)-h_N(t)}\le C\epsilon.
\end{equation} 
Moreover, for $r\in (0, 1/10)$ and $\epsilon\in (0, r/10)$ the following improved estimate holds:
\begin{equation}\label{eq:distancedensitiesImproved}
\abs{h(t)-h_N(t)}\le C r^{N-2}\epsilon, \quad \text{for all } t\in ([0, r]\cup[\pi-r, D]).
\end{equation} 
\end{corollary}
\begin{proof}
The validity of \eqref{eq:distancedensities} follows from \eqref{eq:distdens1} taking into account the Lipschitz continuity of $h_N$ and the asymptotic expansions of 
\begin{equation*}
\frac{\omega_N}{\omega_N\lambda_D+\epsilon}\quad\text{and}\quad\frac{\omega_N}{\omega_N-\epsilon},
\end{equation*}
as $\epsilon\to 0$.  The improved estimate \eqref{eq:distancedensitiesImproved} on  $([0, r]\cup[\pi-r, D])$ follows analogously from  \eqref{eq:distdens1} and the mean value theorem.
\end{proof}

Armed with \autoref{cor:estdens} we can prove that, if $D\in (0,\pi)$ is close to $\pi$, then the integrals of the functions $\sin$ and $\cos$ (and of any bounded function, more in general) with respect to a $\CD(N-1,N)$ density $h$ defined on $[0,D]$ do not differ much from the value of the corresponding integrals computed with respect to the model density  $h_N$.

\begin{corollary}\label{cor:estintegrals}
Let $f:[0,\pi]\to[-1,1]$ be Borel measurable. Denote $\meas(dt):=h(t) \L^{1}(dt)$ and $\meas_N(dt):=h_{N}(t) \L^{1}(dt)$. Under the assumptions of Proposition \ref{prop:hCDN-1N}, there exist a constant $C=C(N)>0$ and $\epsilon_0>0$ with the following property:  if $\epsilon\in [0,\epsilon_0]$  then 
\begin{equation}\label{eq:estintegrals}
\abs{\int_0^Df(t)\,\meas(dt)-\int_0^{\pi}f(t) \,\meas_N(dt)} \le C\epsilon .
\end{equation}
 Moreover,  for and any  $r\in (0, 1/10)$ and $\epsilon\in (0, r/10)$ the following improved estimate holds
\begin{equation} \label{eq:estintegralsimproved}
\abs{\int_0^r f(t)\,\meas(dt)-\int_0^{r}f(t) \,\meas_N(dt)}+  \abs{\int_{\pi-r} ^{D} f(t)\,\meas(dt)-\int_{\pi-r}^{\pi}f(t) \,\meas_N(dt)} \le C\epsilon r^{N-1} .
\end{equation}

\end{corollary}

\begin{proof}
The conclusion follows from \autoref{cor:estdens} just by integrating on $[0,D]$ and taking into account that $|\int_{D}^{\pi} f \, \meas_{N} |\leq   C\epsilon^{N}$
\end{proof}

\subsection{Localization and $L^{1}$-optimal transportation}\label{Ss:localization}

The localization technique has its roots in a work of  Payne-Weinberger \cite{PW} and has been developed by Gromov-Milman \cite{GrMi}, Lov\'asz-Simonovits \cite{LoSi} and Kannan-Lov\'asz-Simonovits \cite{KaLoSi} in the setting of Euclidean spaces, spheres and Hilbert spaces. The basic idea is to  reduce an $n$-dimensional problem, via tools of convex geometry,
to lower dimensional problems which are easier to handle.  In the aforementioned papers, the symmetries of the spaces were heavily used to obtain such a dimensional reduction, typically via iterative bisections.
In the recent paper \cite{klartag}, Klartag found a bridge  between $L^1$-optimal transportation problems and the localization techinque yielding the localization theorem in the framework of smooth Riemannian manifolds.
Inspired by this approach, the first and the second author in \cite{CavallettiMondino17a}
proved a localization theorem for essentially non-branching metric measure spaces verifying the $\CD(K,N)$ condition. Before stating the result it is worth recalling some basics about the disintegration of a measure associated to a partition (for a comprehensive treatment see the monograph by Fremlin \cite{Fre:measuretheory4}; for a discussion closer to the spirit of this paper see \cite{biacar:cmono};  for a one-page summary see \cite[Appendix B]{CMM}).
\\

Given a measure space $(X,\mathscr{X},\mm)$, suppose
a \emph{partition} of $X$ is given into \emph{disjoint} sets $\{ X_{q}\}_{q \in Q}$ so that $X = \cup_{q \in Q} X_q$.
Here $Q$ is the set of indices and $\QQ : X \to Q$ is the quotient map, i.e.
$$
q = \QQ(x) \iff x \in X_{q}.
$$
We endow $Q$ with the \emph{push forward $\sigma$-algebra} $\mathcal{Q}$ of $\mathscr{X}$:
$$
C \in \mathcal{Q} \quad \Longleftrightarrow \quad \QQ^{-1}(C) \in \mathscr{X},
$$
i.e. the biggest $\sigma$-algebra on $Q$ such that $\QQ$ is measurable. Moreover, the push forward measure  $\qq := \QQ_\sharp \, \mm$ defines a natural measure $\qq$ on $(Q,\mathcal{Q})$. The triple  $(Q, \mathcal{Q}, \qq)$ is called  \emph{the quotient measure space}.

\begin{definition}[Consistent and Strongly Consistent Disintegration]
\label{defi:dis}
A \emph{disintegration} of $\mm$ \emph{consistent with the partition} is a map
$$
Q \ni q \longmapsto \mm_{q} \in \mathcal{P}(X,\mathscr{X})
$$
such that the following requirements hold:
\begin{enumerate}
\item  for all $B \in \mathscr{X}$, the map $q \mapsto \mm_{q}(B)$ is $\qq$-measurable;
\item for all $B \in \mathscr{X}$ and $C \in \mathcal{Q}$, the following consistency condition holds:
$$
\mm \left(B \cap \QQ^{-1}(C) \right) = \int_{C} \mm_{q}(B)\, \qq(dq).
$$
\end{enumerate}
A disintegration of $\mm$ is called \emph{strongly consistent} if in addition:
\begin{enumerate}
\item[3.] for $\qq$-a.e. $q \in Q$, $\mm_q$ is concentrated on $X_{q} = \QQ^{-1}(q)$.
\end{enumerate}
\end{definition}

In the next theorem, for $\qq$-a.e. $q\in Q$, the equivalence class $X_{q}$ is a geodesic in $X$. With a slight abuse of notation  $X_{q}$ denotes also the the arc-length parametrization on a real interval of the corresponding geodesic, i.e. it is  a map from a real interval  with image $X_{q}$.  We will use the following terminology: $q \mapsto \mm_{q}$ is a $\CD(K,N)$ disintegration if
for $\qq$-a.e. $q \in Q$, $\mm_{q} = h_{q} \mathcal{H}^{1}\llcorner_{X_{q}}$, where $\mathcal{H}^{1}$ denotes the one-dimensional Hausdorff measure
and $h_{q} \circ X_{q}$  is a $\CD(K,N)$ density, in the sense of \eqref{E:1dCD}.

\begin{theorem}[\cite{CavallettiMondino17a}]\label{T:localize}
Let $(X,\sfd, \mm)$ be an essentially non-branching metric measure space verifying the $\CD(K,N)$ condition for some $K\in \R$ and $N\in [1,\infty)$.
Let $f : X \to \R$ be $\mm$-integrable such that $\int_{X} f\, \mm = 0$ and assume the existence of $x_{0} \in X$ such that $\int_{X} | f(x) |\,  \sfd(x,x_{0})\, \mm(dx)< \infty$.
\medskip

Then the space $X$ admits a partition 
$\{ X_{q} \}_{q \in Q}$ and a corresponding (strongly consistent) disintegration of 
$\mm$, $\{\mm_{q} \}_{q \in Q}$ such that:

\begin{itemize}
\item For any $\mm$-measurable set $B \subset \mathcal{T}$ it holds
$$
\mm(B) = \int_{Q} \mm_{q}(B) \, \qq(dq),
$$
where $\qq$ is a probability measure over $Q$ defined on the quotient $\sigma$-algebra $\mathcal{Q}$.
\medskip
\item For $\qq$-almost every $q \in Q$, the set $X_{q}$ is a geodesic (possibly of zero length) and $\mm_{q}$ is supported on it.
Moreover $q \mapsto \mm_{q}$ is a $\CD(K,N)$ disintegration.
\medskip
\item For $\qq$-almost every $q \in Q$, it holds $\int_{X_{q}} f \, \mm_{q} = 0$. 
\end{itemize}
\end{theorem}

In \autoref{T:localize} we can also distinguish the set of 
$X_{\alpha}$ having positive length, whose union forms the so-called \emph{transport set} denoted by 
$\mathcal{T}$, 
from the ones having zero length, i.e. points, whose union we usually denote with $Z$,  
so to have a decomposition of $X$ into $\mathcal{T}$ and $Z$. 
The last point of \autoref{T:localize} implies then that $\mm$-a.e. $f \equiv 0$ on $Z$.

Following  the approach of \cite{klartag}, \autoref{T:localize} has been proven in \cite{CavallettiMondino17a} studying the following optimal transportation problem. Let
$\mu_{0} : = f^{+} \mm$ and $\mu_{1} : = f^{-}\mm$, where $f^{\pm}$ denote the positive and the negative part of $f$ respectively, and study the
$L^{1}$-optimal transport problem associated with it
\begin{equation}\label{eq:defL1OT}
\inf \left\{ \int_{X\times X} \sfd(x,y) \, \pi(dxdy) \colon \pi \in \mathcal{P}(X\times X), \ (P_{1})_{\sharp}\pi = \mu_{0}, (P_{2})_{\sharp}\pi = \mu_{1} \right\}.
\end{equation}
Then the relevant object to study is given by the dual formulation of the previous minimization problem.
By the summability properties of $f$ (see the hypotheses of \autoref{T:localize}), there exists a $1$-Lipschitz function $\f : X \to \R$  such that
$\pi$ is a minimizer in \eqref{eq:defL1OT} if and only if $\pi(\Gamma) = 1$, where
$$
\Gamma : = \{ (x,y) \in X \times X  \colon \f(x) - \f(y) = \sfd(x,y)\}
$$
is the naturally associated $\sfd$-cyclically monotone set, i.e. for any $(x_{1},y_{1}), \dots, (x_{n},y_{n}) \in \Gamma$ it holds
$$
\sum_{i = 1}^{n} \sfd(x_{i},y_{i}) \leq \sum_{i = 1}^{n} \sfd(x_{i},y_{i+1}),  \qquad y_{n+1} = y_{1},
$$
for any $n \in \N$.
The set $\Gamma$ induces a partial order relation whose maximal chains produce a partition (up to an $\mm$-negligible subset)  of  the set  $\mathcal{T}\subset X$ appearing in the statement of \autoref{T:localize}, made of one dimensional subsets. For a summary of the constructions see \cite[Section 2.5]{CMM}, for more details see \cite{CavallettiMondino17a,CMi}.


\section{One dimensional estimates}\label{S:1d}
The goal of this section is to give a self-contained presentation of the 1-dimensional estimates we will use throughout the paper. 

\subsection{Berard-Besson-Gallot explicit lower bound on the  model Isoperimetric profile}

For $N>1$, let 
\begin{equation}\label{def:mmN}
\omega_N:=\int_0^{\pi}\left(\sin t\right)^{N-1}\di t  \quad \text{and} \quad \mm_{N}:=\frac{1}{\omega_{N}} (\sin t)^{N-1} \L^{1}(dt)\llcorner[0,\pi].
\end{equation}
From now on fix $D\in (0,\pi)$. For $b\in [0,\pi-D]$ and $v\in [0,1]$,  let $R(b,v)\in [b,\pi]$ be uniquely defined by the equation
\begin{equation}\label{eq:defRbv}
\int_b^{R(b,v)}\left(\sin t\right)^{N-1}\di t=v\int_b^{b+D}\left(\sin t\right)^{N-1}\di t.
\end{equation}
Set
\begin{equation}\label{eq:exprlower}
\mathcal{I}_{N,D}(v):=\inf\left\lbrace g(b,v):\, b\in[0,\pi-D] \right\rbrace, 
\end{equation}
where 
\begin{equation}
g(b,v):=\frac{\left[\sin\left(R(b,v) \right)  \right]^{N-1}}{\int_b^{b+D}\left(\sin t\right)^{N-1}\di t}.
\end{equation}
To keep notation short, we also set $\mathcal{I}_{N}:=\mathcal{I}_{N,\pi}$.  Notice that  $\mathcal{I}_{N}$ is the isoperimetric profile of  ${\mathbb S}^{N}$,  for integer $N$. We refer to \autoref{S:Obatadiam} for a brief discussion about the isoperimetric profile;   note also that $\mathcal{I}_{N,D}$ is  the model isoperimetric profile in the L\'evy-Gromov isoperimetric comparison Theorem for spaces with Ricci $\geq N-1$, dimension $\leq N$ and diameter $\leq D$, see \cite[Appendix C]{Gro}, \cite{BerardBessonGallot85, Milman15, CavallettiMondino17a}.

The proof of the next lemma is inspired by, but somewhat different from,  \cite[Appendix 1]{BerardBessonGallot85} and the statement generalises  to arbitrary real $N>1$ the result stated in the reference for integer $N\geq 2$.

\begin{lemma}[Berard-Besson-Gallot explicit isoperimetric lower bound]\label{thm:improvedLevyGromov1D}
Fix $N>1$ and $D\in [0,\pi]$. and let  $\mathcal{I}_{N,D}:[0,1]\to [0,\infty)$ be defined in \eqref{eq:exprlower}. 
Then
\begin{equation}\label{eq:IND/INpigeq}
\frac{\mathcal{I}_{N,D}(v)}{\mathcal{I}_{N}(v)}\ge   \left(\frac{\int_0^{\frac{\pi}{2}}\left(\cos t\right)^{N-1}\di t}{\int_0^{\frac{D}{2}}\left(\cos t\right)^{N-1}\di t}\right)^{\frac{1}{N}}=:C_{N,D}\ge 1, \quad \forall v\in (0,1)
\end{equation}
\end{lemma}

\begin{proof}
Let $v'\in (0,1)$ and $f:[0,\pi-D]\times(0,1)\to[0,+\infty)$ be defined by
\begin{equation}\label{eq:defv'f}
v':=\frac{1}{\omega_N}\int_0^{R(b,v)}\left(\sin t\right)^{N-1}\di t\quad \text{and}\quad f(b,v):=\frac{g(b,v)}{\mathcal{I}_N(v)}.
\end{equation}
In particular 
\begin{equation}
\mathcal{I}_N(v')=\frac{1}{\omega_N}\left[\sin R(b,v)\right]^{N-1}
\end{equation}
and therefore
\begin{equation}\label{eq:exprf}
f(b,v)=\omega_N\left(\int_b^{b+D}\left(\sin t\right)^{N-1}\di t\right)^{-1}\frac{\mathcal{I}_N(v')}{\mathcal{I}_N(v)}.
\end{equation}
Thanks to the explicit expression of the isoperimetric profile $\mathcal{I}_N$ it is possible to compute
\begin{equation}\label{eq:diffequmodel}
\left(\mathcal{I}_N^{\frac{N}{N-1}}\right)''\mathcal{I}_{N}^{\frac{N-2}{N-1}}=-N.
\end{equation} 
In particular it follows from \eqref{eq:diffequmodel} that $\mathcal{I}_N^{\frac{N}{N-1}}$ is concave on $(0,1)$.
\\We now distinguish two cases:  $v'\le v$ and  $v'> v$.

\textbf{Case 1}: $v'\le v$.
\\First observe that
\begin{equation}\label{eq:int1}
\omega_Nv'= \int_0^{R(b,v)}\left(\sin t\right)^{N-1}\di t \ge\int_b^{R(b,v)}\left(\sin t\right)^{N-1}\di t=v\int_b^{b+D}\left(\sin t\right)^{N-1}\di t.
\end{equation}
The concavity observed above together with \eqref{eq:int1} give that
\begin{equation*}
\frac{\mathcal{I}_N(v')}{\mathcal{I}_N(v)}\ge \left(\frac{v'}{v}\right)^{1-\frac{1}{N}}\ge \left(\omega_N^{-1}\int_b^{b+D}\left(\sin t\right)^{N-1}\di t\right)^{1-\frac{1}{N}}.
\end{equation*}
Hence, taking into account \eqref{eq:exprf}, we obtain
\begin{equation}\label{eq:firstcase}
f(b,v)\ge \omega_N^{\frac{1}{N}}\left(\int_b^{b+D}\left(\sin t\right)^{N-1}\di t\right)^{-\frac{1}{N}}.
\end{equation}

\textbf{Case 2}: $v'> v$.
\\A change of variables in the definition of $R$ easily yields
\begin{equation*}
R(\pi-b-D,1-v)=\pi-R(b,v)
\end{equation*}
and therefore
\begin{equation}\label{eq:simm}
f(b,v)=f(\pi-b-D,1-v).
\end{equation}
Moreover
\begin{equation*}
\int_0^{R(\pi-b-D,1-v)}\left(\sin t\right)^{N-1}\di t=\int_{R(b,v)}^{\pi}\left(\sin t\right)^{N-1}\di t=(1-v')\omega_N,
\end{equation*}
hence 
\begin{equation}\label{eq:int2}
f(\pi-b-D,1-v)=\omega_N\left(\int_b^{b+D}\left(\sin t\right)^{N-1}\right)^{-1}\frac{\mathcal{I}_N(1-v')}{\mathcal{I}_N(1-v)}.
\end{equation}
Next we observe that, as in the previous case, the concavity of $\mathcal{I}_N^{\frac{N}{N-1}}$ yields
\begin{equation}\label{eq:consconc2}
\frac{\mathcal{I}_N(1-v')}{\mathcal{I}_N(1-v)}\ge \left(\frac{1-v'}{1-v}\right)^{1-\frac{1}{N}}.
\end{equation}
Moreover, it holds
\begin{equation}\label{eq:int3}
\omega_N(1-v')=  \int_{R(b,v)}^{\pi}\left(\sin t\right)^{N-1}\di t  \ge \int_{R(b,v)}^{b+D}\left(\sin t\right)^{N-1}\di t=\left(1-v\right)\int_b^{b+D}\left(\sin t\right)^{N-1}\di t.
\end{equation}
Combining \eqref{eq:int2}, \eqref{eq:consconc2} and \eqref{eq:int3} and taking into account \eqref{eq:simm}, we get
\begin{equation}\label{eq:case2}
f(b,v)\ge \omega_N^{\frac{1}{N}}\left(\int_b^{b+D}\left(\sin t\right)^{N-1}\di t\right)^{-\frac{1}{N}}.
\end{equation}

It is now sufficient to observe that the function $x\mapsto\int_x^{x+D}\left(\sin t\right)^{N-1}\di t$ attains its maximum at $x=\frac{\pi}{2}-\frac{D}{2}$ in order to obtain from \eqref{eq:firstcase}, \eqref{eq:case2}, \eqref{eq:defv'f} and \eqref{eq:exprlower} that
\begin{equation*}
\frac{\mathcal{I}_{N,D}(v)}{\mathcal{I}_N(v)}\ge \left(\frac{\omega_N}{\int_{\frac{\pi}{2}-\frac{D}{2}}^{\frac{\pi}{2}+\frac{D}{2}}(\sin t)^{N-1}\di t}\right)^{\frac{1}{N}}=C_{N,D}, \quad \forall v\in(0,1).
\end{equation*}
Above, the last identity follows from the expression for $C_{N,D}$ introduced in \eqref{eq:IND/INpigeq} thanks to the identity $\cos(\pi/2-x)=\sin(x)$ and a change of variables.
\end{proof}

Let us study the behaviour of $C_{N,D}$ in the asymptotic $D\to \pi$.

\begin{lemma}\label{prop:asymptoticC}
It holds that
\begin{equation}\label{CND2Dtopi}
\lim_{D\to \pi} \frac{(\pi-D)^{N}}{C_{N,D}^{2}-1}=  2^{N-1}N^{2} \int_0^{\frac{\pi}{2}} \left(\cos t\right)^{N-1}\di t.
\end{equation}
Hence there exist $\bar{C}=\bar{C}(N)>0$ and $\bar{D}=\bar{D}(N)<\pi$ such that
\begin{equation}\label{eq:estimateC}
C_{N,D}^{2}-1\ge \bar{C}(\pi-D)^{N}, \quad \forall D\in [\bar{D}, \pi].
\end{equation}
\end{lemma}

\begin{proof}
Recalling the expression of $C_{N,D}$ from \eqref{eq:IND/INpigeq}, we have
\begin{align*}
C_{N,D}^{2}-1=&\left(\frac{\int_0^{\frac{\pi}{2}}\left(\cos t\right)^{N-1}\di t}{\int_0^{\frac{D}{2}}\left(\cos t\right)^{N-1}\di t}\right)^{\frac{2}{N}}-1
=\left(1+\frac{\int_{\frac{D}{2}}^{\frac{\pi}{2}}\left(\cos t\right)^{N-1}\di t}{\int_0^{\frac{D}{2}}\left(\cos t\right)^{N-1}\di t}\right)^{\frac{2}{N}}-1.
\end{align*}
Now, as $D\to \pi$, we have the expansion 
\begin{equation*}
\int_{\frac{D}{2}}^{\frac{\pi}{2}}\left(\cos t\right)^{N-1}\di t=\int_{0}^{\frac{\pi}{2}-\frac{D}{2}}\left(\sin t\right)^{N-1}\di t\sim \int_0^{\frac{\pi}{2}-\frac{D}{2}}s^{N-1}\di s\sim \frac{1}{N} \left(\frac{\pi}{2}-\frac{D}{2} \right)^{N}.
\end{equation*}
Taking into account the asymptotic $(1+x)^{\beta}-1\sim\beta x$, we obtain \eqref{CND2Dtopi}.
\\The second conclusion in the statement easily follows from the first one.
\end{proof}


\subsection{Spectral gap and diameter}

Building on top of the lower bound of the  isoperimetric  profile obtained in \autoref{thm:improvedLevyGromov1D}, we next obtain  a quantitative spectral gap inequality for Neumann boundary conditions in terms diameters.

The analogous result in the case of smooth Riemannian manifolds was established in \cite[Theorem B]{Croke82} building upon a quantitative improvement of the Lévy-Gromov inequality and on \cite{BerardMeyer82} (see also \cite[Corollary 17]{BerardBessonGallot85}).
The usual strategy to show the improved Neumann spectral gap inequality is based on the observation that a Neumann first eigenfunction of the Laplacian $f$ is a Dirichlet first eigenfunction of the Laplacian on the domains $\set{f>0}$ and $\set{f<0}$ (cf. \cite[Lemma 3.2]{Matei00}). The improved Dirichlet spectral gap inequality is then obtained by rearrangement starting from the isoperimetric inequality.

\begin{proposition}[1-Dimensional Quantitative Obata's Theorem on the diameter]\label{thm:improvedspectralNeumann1D}
Let $(I,\dist_{eucl},\meas)$ be a one dimensional $\CD(N-1,N)$ m.m.s. with $\diam(I)\leq D$. Then
\begin{equation}
\frac{ \lambda^{1,2}_{(I,\dist_{eucl},\meas)}}{N}\ge C_{N,D}^2
= \left(\frac{\int_0^{\frac{\pi}{2}}\left(\cos t\right)^{N-1}\di t}{\int_0^{\frac{D}{2}}\left(\cos t\right)^{N-1}\di t}\right)^{\frac{2}{N}},
\end{equation}
where $C_{N,D}$ was defined in \eqref{eq:IND/INpigeq}.
\\In particular, there exists a constant $C_{N}>0$ (more precisely one can choose $C_{N}=\bar{C} N$ where  $\bar{C}$ was defined in  \autoref{prop:asymptoticC}) such that
\begin{equation}\label{eq:1DQuantObata}
C_{N}(\pi-\diam(I))^N\le  \lambda^{1,2}_{(I,\dist_{eucl},\meas)}-N.
\end{equation} 
\end{proposition}

\begin{proof}
From  \cite{BakryQian00} (see also  \cite[Section 4.1]{CavallettiMondino17b} for the regularization procedure) we know that  $\lambda^{1,2}_{(I,\dist_{eucl},\meas)}\geq  \lambda^{1,2}_{N,D}$ where $\lambda^{1,2}_{N,D}$ is the first solution $\lambda>0$ of the eigenvalue problem
\begin{equation}\label{eq:eigenvalueproblem}
\ddot{w}+(N-1)\tan (t)\dot{w}+\lambda w=0,
\end{equation}
on $[-D/2,D/2]$ with Neumann boundary conditions.
The eigenfunction associated to the first eigenvalue in \eqref{eq:eigenvalueproblem} is unique, up to a multiplicative constant. Therefore, denoting it by $w_{N,D}:[-D/2,D/2]\to (-\infty,+\infty)$, it holds $w_{N,D}(-x)=-w_{N,D}(x)$ for any $x\in[-D/2,D/2]$ as a consequence of the symmetry of \eqref{eq:eigenvalueproblem}. In particular $w_{N,D}(0)=0$. 
Let
$$\meas_{N,D}:=\Lambda_{N,D}\left(\cos t\right)^{N-1}\L^1\res[-D/2,D/2],$$ 
with $\Lambda_{N,D}$ such that $\meas_{N,D}$ is a probability measure. Note that $\left([-D/2,D/2],\dist_{eucl},\meas_{N,D}\right)$ is a $\CD(N-1,N)$ m.m.s. with diameter equal to $D$ and $\meas_{N,D}([-D/2,0])=\meas_{N,D}([0, D/2])=1/2$. Hence
\begin{equation*}
\lambda^{1,2}_{N,D}=\frac{\int_{-D/2}^{D/2}\abs{w_{N,D}'}^2\mm_{N,D}}{\int_{-D/2}^{D/2}\abs{w_{N,D}}^2\mm_{N,D}}=\frac{\int_{0}^{D/2}\abs{w_{N,D}'}^2\mm_{N,D}}{\int_{0}^{D/2}\abs{w_{N,D}}^2\mm_{N,D}}\geq \lambda^{1,2, \mathcal{D}}_{N,D}(1/2),
\end{equation*}  
where $\lambda^{1,2, \mathcal{D}}_{N,D}(1/2)$ is the least first eigenvalue of the Laplacian with Dirichlet boundary conditions on one extremum for intervals of volume $1/2$ in  $\left([-D/2,D/2],\dist_{eucl},\meas_{N,D}\right)$.
\\Moreover a co-area argument (see for instance \cite[Corollary 17]{BerardBessonGallot85} or \cite[Proposition 3.13]{MondinoSemola18}) using \autoref{thm:improvedLevyGromov1D}  gives
\begin{equation*}
\lambda^{1,2, \mathcal{D}}_{N,D}(1/2)\ge C_{N,D}^2 \, \lambda^{1,2, \mathcal{D}}_{N,\pi}(1/2).
\end{equation*}
Recalling that $ \lambda^{1,2, \mathcal{D}}_{N,\pi}(1/2)=\lambda^{1,2}_{N, \pi}=N$ (see for instance  \cite{BakryQian00}), we conclude that 
\begin{equation}\label{eq:la12NDC2ND}
\lambda^{1,2}_{(I,\dist_{eucl},\meas)}\geq \lambda^{1,2}_{N,D} \geq \lambda^{1,2, \mathcal{D}}_{N,D}(1/2)  \geq N\,C_{N,D}^2.
\end{equation}
The second part of the statement follows by choosing $D=\diam(I)$ and applying  \autoref{prop:asymptoticC}.
\end{proof}

A converse of the inequality proved in \autoref{thm:improvedspectralNeumann1D} 
can be obtained as follows.
 
\begin{lemma}\label{lem:BigDLapi}
For any $N>1$ there exists $C=C(N)>0$  such that if $([0,D],\sfd_{eucl}, \mm)$ is a one dimensional $\CD(N-1,N)$ m.m.s. with $D \geq \pi-\epsilon$ then
\begin{equation*}
\left|\lambda^{1,2}_{([0,D],\sfd_{eucl}, \mm)}-N \right|  \leq  C \epsilon.
\end{equation*} 
\end{lemma}

\begin{proof}
By Lichnerowicz spectral gap we already know that $\lambda^{1,2}_{([0,D],\sfd_{eucl}, \mm)}\geq N$. It is therefore enough to prove the existence of $u\in \Lip([0,D])$ such that 
\begin{equation}\label{eq:claimClosela12}
\|u\|_{L^{2}([0,D],\mm)}=1, \quad \int_{[0,D]} u \, \mm =0, \quad \int_{[0,D]}  |u'|^{2} \mm \leq N+ C_{N} \epsilon.
\end{equation}
Setting $u^{*}_{N}(t):= \sqrt{N+1} \cos(t)$ and using \autoref{cor:estintegrals} we get 
\begin{equation}\label{eq:u*Nmm}
\left|\int_{[0,D]} u^{*}_{N} \, \mm \right| \leq C_{N} \epsilon, \quad  \left|1-\int_{[0,D]} |u^{*}_{N}|^{2} \, \mm \right| \leq C_{N} \epsilon,  \quad  \int_{[0,D]} |(u^{*}_{N})'|^{2} \, \mm  \leq  N +C_{N} \epsilon.
\end{equation}
Let $v=u^{*}_{N}-\int_{[0,D]} u^{*}_{N}\, \mm$ and   $c_{v}:= \|v\|_{L^{2}([0,D], \mm)}$. 
Using the estimates \eqref{eq:u*Nmm}, it is straightforward to check that $u= \frac{1}{c_{v}} v$ satisfies \eqref{eq:claimClosela12}.
\end{proof}


\subsection{Spectral gap and shape of eigenfunctions}
Next we establish some basic estimates on eigenfunctions which will be useful later. 

Given a one dimensional $\CD(K,N)$ space $(I, \sfd_{eucl}, \mm)$, we know that we can write $\mm(dt)=h \L^{1}(dt)$ for some $\CD(K,N)$ density $h$.
We start by recalling the definition and basic properties of the  Laplace operator $\Delta$.  
A function $u\in W^{1,2}(I,  \mm)$ is said to be in the domain of $\Delta$, and we write $u \in \dom(\Delta)$ if for every $\varphi\in C^{\infty}_{c}(I)$ it holds
$$
\left|\int_{I} u' \varphi' \, \mm \right| \leq C_{u}  \|\varphi\|_{L^{2}(I,\mm)},
$$
for some $C_{u}\geq 0$ depending on $u$. In this case, by Riesz Theorem, there exists a function $\Delta u\in L^{2}(I,\mm)$ such that
$$
- \int_{I} u' \varphi' \, \mm =  \int_{I}  \Delta u \, \varphi \, \mm.
$$
It is readily seen that the operator $\dom(\Delta)\ni u\mapsto \Delta u\in  L^{2}(I,\mm)$ is linear. 

Moreover, using the properties of $\CD(K,N)$ densities recalled at the beginning of the section, it holds that every $u\in \dom(\Delta)$ is twice differentiable $\L^{1}$-a.e. on $I$ and 
\begin{equation}\label{eq:formulaDelta}
\Delta u = u''+ (\log h)' u', \quad \L^{1}\text{-a.e. on $I$, }\forall  u\in \dom(\Delta).
\end{equation}

\begin{proposition}\label{cor:smallhessian1d} 
Let $(I,\dist_{eucl},\meas)$ be a one dimensional $\CD(N-1,N)$ m.m.s.. Then there exists a constant $C=C(N)>0$ such that, if $u$ is an eigenfunction of the Laplacian on $(I,\dist_{eucl},\meas)$ associated to an  eigenvalue  $\lambda\in [N, 2N]$ and with $\norm{u}_2=1$, then $u\in W^{2,2}_{\loc}\left(I,\dist_{eucl},\L^1\right)$ and
\begin{equation}\label{eq:estonedim}
\norm{u''+u}_{L^2(\meas)}\le C\left(\lambda-N\right)^{\frac{1}{2}}.
\end{equation}
\end{proposition}

\begin{proof}
\textbf{Step 1}.  
We claim that it holds 
\begin{equation}\label{eq:u''Deltau}
\int_{I} \left(u''- \frac{1}{N} \Delta u \right)^{2} \mm \leq  \int_{I} \left(  \frac{N-1}{N}  (\Delta u)^{2}- (N-1) (u')^{2} \right) \mm .
\end{equation}
Since by assumption $u\in W^{1,2}(I, \sfd_{eucl}, \mm)$  is an eigenfunction we have  $-\Delta u\in W^{1,2}(I, \sfd_{eucl}, \mm)$ as well. Thus we can  define the $\Gamma_{2}$ operator as
\begin{equation}\label{eq:defGamma2EF}
\Gamma_{2}(u; \varphi):= 
\int_{I}  \left(\frac{1}{2} (u')^{2} \Delta \varphi 
- (\Delta u)' \, u'    \varphi \right) \, \mm, 
\end{equation}
for all $\varphi\in L^{\infty}(I, \mm)$ with $\Delta \varphi\in L^{\infty}(I, \mm)$.
Using that  $h$ satisfies \eqref{eq:logh''-K}, a manipulation via integration by parts gives that for all $\varphi\geq 0$ as above it holds:
\begin{equation}\label{eq:proofGamma2u''}
\Gamma_{2}(u; \varphi)\geq \int_{I} \left[ (u'')^{2} + (N-1) (u')^{2} + \frac{1}{N-1} (\Delta u-u'')^{2} \right] \varphi \, \mm.
\end{equation}
By direct computations,  one can check  that 
\begin{align}\label{eq:IdentityN-1u'2}
(u'')^{2} +& (N-1) (u')^{2} + \frac{1}{N-1} (\Delta u-u'')^{2} \nonumber  
\\
= &~ (N-1) (u')^{2} + \left(  u''-\frac{1}{N} \Delta u\right)^{2} +\frac{1}{N} (\Delta u)^{2}
 +\frac{1}{N-1} \left( u''-\frac{1}{N} \Delta u \right)^{2},\quad \mm\text{-a.e.}.
\end{align}
Plugging \eqref{eq:IdentityN-1u'2} into \eqref{eq:proofGamma2u''} gives
$$
\Gamma_{2}(u;\varphi)\geq  \int_{I} \left[  (N-1) (u')^{2} + \left(  u''-\frac{1}{N} \Delta u\right)^{2}  +\frac{1}{N} (\Delta u)^{2} \right] \, \varphi \,  \mm.
$$
Choosing $\varphi \equiv1$ yields \eqref{eq:u''Deltau}.
\\

\textbf{Step 2}. \\
Inserting the eigenvalue relation  $\lambda u= - \Delta u$ into \eqref{eq:u''Deltau}, we obtain
\begin{equation}\label{eq:u''DeltauPfula}
\int_{I} \left(u''+ \frac{\lambda}{N} u \right)^{2} \mm \leq  \int_{I}  \left( \frac{N-1}{N}  (\lambda u)^{2}- (N-1) (u')^{2}  \right) \mm=   \frac{N-1}{N} \lambda(\lambda-N)   \int_{I}     u^{2} \mm.
\end{equation}
Eventually, 
\begin{align}
\int_{I} \left(u''+  u \right)^{2} \mm & \leq 2\int_{I} \left| u''+\frac{\lambda}{N} u \right|^{2} \mm +2  \int_{I} \left|\frac{\lambda-N}{N} u \right|^{2} \mm \nonumber \\
& \leq 2\left(  \frac{N-1}{N} \lambda(\lambda-N)  + \frac{(\lambda-N)^{2}}{N^{2}} \right) \int_{I} u^{2} \mm \nonumber \\
&\leq C(N)  (\lambda-N) \int_{I} u^{2} \mm, \nonumber
\end{align}
where, in the last estimate, we used the assumption $\lambda \leq 2N$.
\end{proof}

The aim of the remaining part of this section is to prove  \autoref{thm:mainthm1d} stating roughly  that, on any one dimensional $\CD(N-1,N)$ m.m.s. $(I,\sfd_{eucl}, \mm)$, a function $u:I\to \R$ whose $2$-Rayleigh quotient is close to $N$ (the optimal one on the model $(N-1,N)$-space) and with $L^2$-norm equal to one, is $W^{1,2}$-close  to the (normalized) cosine of the distance from one of the extrema of the interval, in quantitative terms.

The conclusion of \autoref{thm:mainthm1d} will be achieved through some intermediate steps. First we  estimate the  $W^{1,2}$-closeness of a first eigenfunction  $u^{*}$ for $(I,\sfd_{eucl}, \mm)$ 
with the cosine of the distance from one of the extremes of the segment, see \autoref{prop:estdisteigcos}. Then, we bound the $W^{1,2}$-closeness of the function $u$ from  $u^{*}$ (or $- u^{*}$), see  \autoref{prop:deficitestimatesdist1}.

Let us observe that
\begin{equation}\label{eq:defcN}
\norm{\cos(\cdot)}_{L^2(\meas_N)}=\frac{1}{\sqrt{N+1}},
\end{equation}
and, by symmetry,
\begin{equation}\label{eq:intcosmN0}
\int_{[0,\pi]}\cos(t)\,\mm_N(dt)=0.
\end{equation}

\begin{proposition}\label{prop:estdisteigcos}
For every $N>1$ there exist constants $C=C(N)>0$ and 
$\epsilon_0= \epsilon_0(N)>0$ such that for every one 
dimensional $\CD(N-1,N)$ m.m.s. $([0,D],\dist_{eucl},\meas)$ 
and every Neumann eigenfunction $u^{*}$, with $\|u^{*}\|_{L^{2}(\mm)}=1$, of eigenvalue $\lambda \in [N,2N]$
it holds
\begin{equation}
\label{eq:estonedimeigcos}
\min\left\lbrace\norm{u^{*}-\sqrt{N+1}\cos(\cdot)}_{L^2(\meas)},\norm{u^{*}+\sqrt{N+1}\cos(\cdot)}_{L^2(\meas)} \right\rbrace\le C\delta^{\min\left\lbrace 1/2,1/N\right\rbrace}, 
\end{equation}
where $\delta:= \int |\nabla u^{*}|^{2} \mm-N<\epsilon_0$.
Furthermore the conclusion can be improved to $W^{1,2}$-closeness:
\begin{equation}
\min\left\lbrace\norm{\left(u^{*}-\sqrt{N+1}\cos(\cdot)\right)'}_{L^2(\meas)},\norm{\left(u^{*}+\sqrt{N+1}\cos(\cdot)\right)'}_{L^2(\meas)} \right\rbrace\le C\delta^{\min\left\lbrace 1/2,1/N\right\rbrace}. 
\end{equation}

\end{proposition}

\begin{proof}
Let $h:[0,D]\to[0,+\infty)$ be the density of $\meas$ with respect to $\L^1$ and let $x_0\in(0,D)$ be a maximum point of $h$. In \cite[Lemma A.4]{CMM} it is proved that such a maximum point is unique and that $h$ is strictly increasing on $[0,x_0]$ and strictly decreasing on $[x_0,D]$. 

\smallskip
\textbf{Step 1.} \\ 
In this first step we prove that, given $z\in L^2([0,D],\meas)$, any solution of $v''+v=z$ can be written as
\begin{equation}\label{eq:generalexpression}
v(t)=\int_{x_0}^t\sin(t-s)z(s)\di s+\alpha\sin(t)+\beta\cos(t)
\end{equation}
for some $\alpha,\beta\in \setR$. To this aim, it suffices to prove that 
\begin{equation}\label{eq:defv0}
v_0(t):=\int_{x_0}^{t}\sin(t-s)z(s)\di s
\end{equation}
solves $v''+v=z$. First we observe that $v_0$ is well defined, since the assumption $z\in L^2((0,D),\meas)$ guarantees that $z\in L^1_{\loc}((0,D),\L^1)$ (due to the fact that $h$ is locally bounded from below by a strictly positive constant in the interior of $[0,D]$). The fact that it satisfies $v_0''+v_0=z$ follows from an elementary computation.

\smallskip
\textbf{Step 2.} \\ 
Next, we prove that the function $v_0$ defined  in \eqref{eq:defv0} satisfies
\begin{equation}\label{eq:estnormv0}
\norm{v_0}_{L^2(\meas)}\le \pi\norm{z}_{L^2(\meas)}.
\end{equation}
Indeed, taking into account that $\abs{\sin}\le 1$, applying the Cauchy-Schwarz inequality, Fubini's Theorem and recalling that $h$ is increasing on $[0,x_0]$ and decreasing on $[x_0,D]$, we can compute 
\begin{align*}
\norm{v_0}^2_{L^2(\meas)}=&\int_0^D\left(\int_{x_0}^t\sin(t-s)z(s)\di s\right)^2h(t)\di t\\
\le&\pi\int_0^Dh(t) \left|\int_{x_0}^tz^2(s)\di s \right|\di t\\
=&\pi\left(\int_0^{x_0}z^2(s)\int_0^sh(t)\di t\di s+\int_{x_0}^D z^2(s)\int_s^Dh(t)\di t\di s\right)\\
\le&\pi^2\left(\int_0^{x_0}z^2(s)h(s)\di s+\int_{x_0}^{D}z^2(s)h(s)\di s \right)=\pi^2\norm{z}^2_{L^2(\meas)}.   
\end{align*}
Let us remark that from \eqref{eq:estnormv0} it follows applying Cauchy-Schwartz inequality that $\norm{v_0}_{L^1(\meas)}\le\pi\norm{z}_{L^2(\meas)}$.

\smallskip
\textbf{Step 3.} \\
Recall from \autoref{thm:improvedspectralNeumann1D} the bound $\pi-D\le C\delta^{1/N}$. Furthermore we know from \eqref{eq:estonedim} that, if $u^{*}$ is as in the assumptions of the statement, then $(u^{*})''+u^{*}=z$ on $[0,D]$ for some function $z$ such that $\norm{z}_{L^2(\meas)}\le C\delta^{1/2}$.
Hence, as proved in \textbf{Step 1.}, $u^{*}$ can be written as
\begin{equation}\label{eq:intexpr}
u^{*}(t)=\int_{x_0}^t\sin(t-s)z(s)\di s+\alpha\sin(t)+\beta\cos(t),
\end{equation}
for some $\alpha,\beta\in\setR$. We want to show that there exists $C=C(N)>0$ such that $|\alpha|+|\beta|\leq C(N)$.\\
 Set $u_0(t):=\int_{x_0}^t\sin(t-s)z(s)\di s$ and recall that, from \textbf{Step 2.},  it holds $\norm{u_0}_{L^2(\meas)}\le C\delta^{1/2}$.
Since by assumption $u^{*}$ has null mean value, integrating  \eqref{eq:intexpr} over $[0,D]$ with respect to $\mm$ gives 
\begin{equation}\label{eq:u*int0}
0=\alpha\int_{[0,D]}\sin(t)\mm(dt)+\beta\int_{[0,D]}\cos(t)\mm(dt)+\int_{[0,D]} u_0(t)\mm(dt).
\end{equation}
From the last remark in {\bf Step 2.} and \autoref{cor:estintegrals}, it follows that
\begin{equation*}
\left(\int_{[0,\pi]} \sin^{N}(t) \di t + O(\delta^{1/N})\right)\alpha  +O(\delta^{1/N}) \beta+ O(\delta^{1/2}) =0,
\end{equation*}
giving that
\begin{equation}\label{eq:estforalphaLinfy}
\alpha= O(\delta^{1/N}) \beta+ O(\delta^{1/2}).
\end{equation}
In order to estimate $\beta$, we compute the $L^2(\meas)$-norm squared both at the left and at the right hand-side of \eqref{eq:intexpr} to obtain
\begin{align}
1=&\norm{u_0}^2_{L^2(\meas)}+\alpha^2\norm{\sin(\cdot)}_{L^2(\meas)}^2+\beta^2\norm{\cos(\cdot)}^2_{L^2(\meas)} \nonumber \\
&+2\alpha\int u_0(t)\sin(t)\mm(dt)+2\beta\int u_0(t)\cos(t)\mm(dt)+2\alpha\beta\int\sin(t)\cos(t)\mm(dt). \label{eq:L2normu*Linfty}
\end{align}
Plugging \eqref{eq:estforalphaLinfy} into \eqref{eq:L2normu*Linfty}, gives
\begin{equation} \label{eq:estforbetaLinfy}
 \left(1+O(\delta)\right) +  O(\delta^{1/N+1/2}) \ \beta+ \left( \int_{[0,\pi]} \cos^{2}(t) \sin^{N-1}(t) \di t + O(\delta^{1/N})\right) \beta^{2}=0,
\end{equation}
yielding $|\beta|\leq C(N)$ and thus, by  \eqref{eq:estforalphaLinfy}, also $|\alpha|\leq C(N)$.

\smallskip
\textbf{Step 4.} \\
Conclusion. In order to get \eqref{eq:estonedimeigcos}, we have to bound $\abs{\alpha}$ and $\min\left\lbrace\abs{\sqrt{N+1}-\beta},\abs{\sqrt{N+1}+\beta}\right\rbrace$ in terms of $\delta$.\\
From \eqref{eq:u*int0}, {\bf Step 3.}, the last remark in {\bf Step 2.} and \autoref{cor:estintegrals} it follows that
\begin{equation}\label{eq:estforalpha}
\abs{\alpha}\le C(\delta^{1/2}+\delta^{1/N})\le C\delta^{\min\left\lbrace 1/N,1/2\right\rbrace },
\end{equation}
up to increase the value of the constant $C$ in the second inequality. Plugging   \eqref{eq:estforalpha} into  \eqref{eq:L2normu*Linfty} gives
\begin{equation*}
1=O(\delta)+O(\delta^{\min\left\lbrace 1,2/N\right\rbrace })+O(\delta^{1/2})+O(\delta^{\min\left\lbrace1/2,1/N \right\rbrace })+\beta^2/(N+1)
\end{equation*}
and therefore
\begin{equation}\label{eq:estbeta}
\abs{1-\frac{\beta^2}{N+1}}=O(\delta^{\min\left\lbrace 1/2,1/N\right\rbrace }).
\end{equation}
From \eqref{eq:estbeta} we easily obtain that 
\begin{equation}\label{eq:estbetadef}
\min\left\lbrace\abs{\sqrt{N+1}-\beta},\abs{\sqrt{N+1}+\beta} \right\rbrace\le C\delta^{\min\left\lbrace 1/4,1/(2N)\right\rbrace }. 
\end{equation}
In case  
$$\abs{\sqrt{N+1}-\beta}=\min\left\lbrace\abs{\sqrt{N+1}-\beta},\abs{\sqrt{N+1}+\beta} \right\rbrace \leq  C\delta^{\min\left\lbrace 1/4,1/(2N)\right\rbrace }$$ 
(respectively $\abs{\sqrt{N+1}+\beta}=\min\left\lbrace\abs{\sqrt{N+1}-\beta},\abs{\sqrt{N+1}+\beta} \right\rbrace \leq C\delta^{\min\left\lbrace 1/4,1/(2N)\right\rbrace }$), it follows that  
\begin{equation}\label{eq:N+1+betageqN+1}
\abs{\sqrt{N+1}+\beta}\geq 2 \sqrt{N+1}- C\delta^{\min\left\lbrace 1/4,1/(2N)\right\rbrace } \geq \sqrt{N+1}, \quad \text{for } \delta\leq \delta_{0}(N).
\end{equation}
(resp. $\abs{\sqrt{N+1}-\beta}\geq \sqrt{N+1}$). Plugging  \eqref{eq:N+1+betageqN+1} back into \eqref{eq:estbeta} gives  $\abs{\sqrt{N+1}-\beta} \leq  C\delta^{\min\left\lbrace 1/2,1/N\right\rbrace }$ (resp. $\abs{\sqrt{N+1}+\beta} \leq  C\delta^{\min\left\lbrace 1/2,1/N\right\rbrace }$). In conclusion,  \eqref{eq:estbeta} and \eqref{eq:estbetadef} can be bootstrapped to give
\begin{equation}\label{eq:estbetadefImproved}
\min\left\lbrace\abs{\sqrt{N+1}-\beta},\abs{\sqrt{N+1}+\beta} \right\rbrace\le C\delta^{\min\left\lbrace 1/2,1/N\right\rbrace }. 
\end{equation}
Combining all these ingredients we can eventually estimate the $L^2(\meas)$-distance between the first Neumann eigenfunction and the normalized cosine. Indeed, assuming without loss of generality that $\abs{\sqrt{N+1}-\beta}\le\abs{\sqrt{N+1}+\beta}$ and taking into account \eqref{eq:estforalpha}, \eqref{eq:estbetadefImproved}, we obtain
\begin{align*}
\norm{u^{*}-\sqrt{N+1}\cos(\cdot)}_{L^2(\meas)}=&\norm{u_0+\alpha\sin(\cdot)+\beta\cos(\cdot)-\sqrt{N+1}\cos(\cdot)}_{L^2(\meas)}\\
\le&\abs{\alpha}\norm{\sin(\cdot)}_{L^2(\meas)}+\norm{u_0}_{L^2(\meas)}+\abs{\beta-\sqrt{N+1}}\norm{\cos(\cdot)}_{L^2(\meas)}\\
\le&  C\delta^{\min\left\lbrace 1/2,1/N\right\rbrace}.
\end{align*}
Finally, we improve the $L^2(\meas)$-closeness to $W^{1,2}(\meas)$-closeness. To this aim, differentiate \eqref{eq:intexpr} to obtain
\begin{equation}\label{eq:derivative}
(u^{*})'(t)=\int_{x_0}^t\cos(t-s)z(s)\di s+\alpha\cos(t)-\beta\sin(t).
\end{equation} 
With computations analogous to the ones used to obtain the bound $\norm{v_0}_2\le\pi\norm{z}_2$ in {\bf Step 2.}, one can prove that, letting $w_0(t):=\int_{x_0}^t\cos(t-s)\di s$, it holds $\norm{w_0}_2\le\pi\norm{z}_2$. The sought estimate for 
\begin{equation*}
\min\left\lbrace\norm{\left(u^{*}-\sqrt{N+1}\cos(\cdot)\right)'}_{L^2(\meas)},\norm{\left(u^{*}+\sqrt{N+1}\cos(\cdot)\right)'}_{L^2(\meas)} \right\rbrace
\end{equation*}
follows taking into account \eqref{eq:estforalpha} and \eqref{eq:estbetadef}.
\end{proof}

We isolate the following corollary which will be useful later in the paper.
\begin{corollary}\label{rm:impr0r1}
Under the assumptions of \autoref{prop:estdisteigcos}, setting $r=\delta^{\gamma/N}$ for some $\gamma\in (0,1)$, it holds
\begin{align}
&\min\left\lbrace \norm{u^*-\sqrt{N+1}\cos(\cdot)}_{W^{1,2}([0,r],\mm)},\norm{u^*+\sqrt{N+1}\cos(\cdot)}_{W^{1,2}([0,r],\mm)}\right\rbrace \nonumber \\
& \qquad \qquad \le C(N)\left(\delta^{1/2}+r^{N/2}\delta^{\min\{1/2,1/N\}} \right). \label{eq:impr0r1}
\end{align}
Moreover, for  $\eta\in (0, r/10)$, 
\begin{align}
&\min\left\lbrace \norm{u^*-\sqrt{N+1}\cos(\cdot)}_{W^{1,2}([r-\eta, r+\eta],\mm)},\norm{u^*+\sqrt{N+1}\cos(\cdot)}_{W^{1,2}([r-\eta, r+\eta],\mm)}\right\rbrace \nonumber \\
& \qquad \qquad \le C(N)\left(\delta^{1/2}+(r^{N-1} \eta)^{1/2}\delta^{\min\{1/2,1/N\}} \right). \label{eq:impr0r11}
\end{align}
\end{corollary} 

\begin{proof}
It is enough to improve the final estimates in \textbf{Step 4.} of the proof of \autoref{prop:estdisteigcos} by using \eqref{eq:estintegralsimproved}:
\begin{align*}
\norm{u^{*}-\sqrt{N+1}\cos(\cdot)}_{L^2([0,r],\meas)}=&\norm{u_0+\alpha\sin(\cdot)+\beta\cos(\cdot)-\sqrt{N+1}\cos(\cdot)}_{L^2([0,r]\meas)}\\
\le&\norm{u_0}_{L^2([0,r],\meas)}+\abs{\alpha}\norm{\sin(\cdot)}_{L^2([0,r],\meas)}+\abs{\beta-\sqrt{N+1}}\norm{\cos(\cdot)}_{L^2([0,r]\meas)} \\
\le&  C \Big(\delta^{1/2} + \delta^{\min\left\lbrace 1/2,1/N\right\rbrace} (\norm{\cos(\cdot)}_{L^2([0,r]\meas_{N})}+C \delta^{1/N} r^{N-1}\})  \Big)  \\
\leq& C\left(\delta^{1/2}+r^{N/2}\delta^{\min\{1/2,1/N\}}\right) .
\end{align*}
The improved estimate for the first derivative and for the domain $[r-\eta, r+\eta]$ is analogous.
\end{proof}

\begin{lemma}\label{prop:improvedpoincarefarfrom}
For any $N>1$ there exist $\bar{D}=\bar{D}(N)<\pi$ and $\alpha=\alpha(N)>0$ such that the following holds.  Let $([0,D],\sfd_{eucl}, \mm)$ be a one dimensional $\CD(N-1,N)$ m.m.s. with $D\ge \bar{D}$ 
and $u^{*}$ any first Neumann eigenfunction, with $\|u^{*}\|_{L^{2}(\mm)}=1$.

Then for any $v\in L^{2}([0,D],\meas)$ with $\norm{v}_{L^2(\mm)}=1$ such that $\abs{\int vu^*\mm}\le1/2$ we have
\begin{equation*}
N+\alpha\le \int_{[0,D]}\abs{v'}^2\mm.
\end{equation*} 
\end{lemma}

\begin{proof}
We argue by contradiction.

Suppose there is a sequence of $\CD(N-1,N)$ measures $\mm_{n}=h_{n} \L^{1}$ 
with  $\supp h_{n}=[0,D_{n}]$ and $D_n\uparrow \pi$ satisfying the following: 
for every $n$ there exists   $v_n\in W^{1,2}([0,D_{n}],\sfd_{eucl},\meas_n)$ 
with  $\norm{v_n}_{L^2(\meas_n)}=1$ such that 
\begin{equation}\label{eq:convergence}
\int_{[0,D_{n}]}\abs{v_n'}^2\mm_n\to N \quad\text{as $n\to\infty$, \; and }  \quad  \abs{\int v_n u^*_n\mm_n}\le\frac{1}{2}, 
\end{equation} 
where $u^*_n$ is a first  Neumann  eigenfunction on $([0,D_{n}], \sfd_{eucl}, h_{n}\L^{1})$, i.e.
\begin{equation}\label{eq:u*nEF}
 \int_{[0,D_{n}]} \left|u^{*}_{n} \right|^{2} \mm_{n}=1, \quad  \int_{[0,D_{n}]} \left|(u^{*}_{n})' \right|^{2} \mm_{n}=\lambda_{n}\to N, 
\end{equation}
where in the last identity we used \eqref{eq:formulaDelta} and the convergence of 
$\lambda_{n}$ to $N$ follows from \autoref{lem:BigDLapi}.

From \autoref{cor:estdens}, the fact that $\supp h_{n}=[0,D_{n}]$ with $D_{n}\uparrow \pi$ implies that $(h_{n})$ (extended to the constant $h(D_{n})$ on $[D_{n},\pi]$) are converging uniformly to the  model 1-dimensional  $\CD(N-1,N)$-density  $h_{N}=\frac{1}{c_{N}'}\sin^{N-1}$ on $[0,\pi]$. In particular, for every $\eta\in (0,\pi/2)$ the densities $h_{n}$ restricted to $[\eta, 1-\eta]$ are bounded above and below by strictly positive constants.

The bounds \eqref{eq:u*nEF} then imply that $u^{*}_{n}$ (resp. $v_{n}$) are uniformly $1/2$-H\"older continuous on  $[\eta, \pi-\eta]$ for every $\eta\in (0,\pi/2)$. 

Thus, by Arzel\'a-Ascoli Theorem combined with a standard diagonal argument, there exists  $u^{*}:[0,\pi]\to \R$ (resp. $v:[0,\pi]\to \R$) and a (non-relabeled for simplicity) subsequence such that $u^{*}_{n}\to u^{*}$ (resp. $v_{n}\to v$) uniformly on $[\eta, \pi-\eta]$ for every $\eta\in (0,\pi/2)$. It is also easy to check that
$$
\int_{[0,\pi]} u^{*}_{n} h_{n} \, \varphi \, \L^{1} \to \int_{[0,\pi]} u^{*} h_{N} \, \varphi \, \L^{1}, \quad \int_{[0,\pi]} v_{n} h_{n} \, \varphi \, \L^{1} \to \int_{[0,\pi]} v h_{N} \, \varphi \, \L^{1} \quad \forall  \varphi\in C([0,\pi]).
$$
Combining the last weak convergence statement with the bounds \eqref{eq:convergence}, \eqref{eq:u*nEF} and with \cite[Theorem 6.3]{GMS2013} gives that 
\begin{align}\label{eq:u*21}
&\|u^{*}\|_{L^{2}([0,\pi], \mm_{N})}=\|v\|_{L^{2}([0,\pi], \mm_{N})}=1, \quad \left|\int_{[0,\pi]} u^{*} v \, \mm_{N} \right| \leq \frac{1}{2}, \nonumber \\
&\int_{[0,\pi]} \left| (u^{*})' \right|^{2} \mm_{N}\leq N, \quad \int_{[0,\pi]} \left|  v' \right|^{2} \mm_{N} \leq N. \nonumber
\end{align}
Therefore, both $u^{*}$ and $v$ are first Neumann eigenfunctions on the model space $([0,\pi], \sfd_{eucl}, \mm_{N})$. However the first eigenfunction is unique up to a sign, thus  it must hold $\left|\int_{[0,\pi]} u^{*} v \, \mm_{N} \right|=1$. Contradiction.
\end{proof}

\begin{corollary}\label{cor:useful}
For every $N>1$ there exists $\beta=\beta(N)>0$ with the following property. Let $(I,\sfd_{eucl},\meas)$ be a one dimensional $\CD(N-1,N)$ m.m.s.  with $\meas(I)=1$ and satisfying
\begin{equation*}
\lambda^{1,2}_{(I,\sfd_{eucl},\meas)}-N<\beta.
\end{equation*}
Then, for any $u\in W^{1,2}(I,\sfd_{eucl},\meas)$ with $\norm{u}_{L^2(\meas)}=1$ and $\abs{\int_I uu^*\,\mm}\le 1/2$, where $u^{*}$ is a first Neumann eigenfunction with $\|u^{*}\|_{L^{2}(\mm)}=1$, it holds
\begin{equation*}
\lambda^{1,2}_{(I,\sfd_{eucl},\meas)}+\beta<\int\abs{u'}^2\mm.
\end{equation*}
\end{corollary}

\begin{proof}
First choose $\beta>0$ sufficiently small so that, by \autoref{thm:improvedspectralNeumann1D},   the diameter of $(I,\sfd_{eucl},\meas)$ is bigger than $\bar{D}$. Then conclude by \autoref{prop:improvedpoincarefarfrom} (and decrease the constant $\beta>0$ if necessary). 
\end{proof}

\begin{proposition}\label{prop:deficitestimatesdist1}
For every $N>1$ there exists $\beta=\beta(N)>0$ with the following property. Let $(I,\sfd_{eucl},\meas)$ be a one dimensional $\CD(N-1,N)$ m.m.s.  with $\meas(I)=1$. Assume there exists $v\in W^{1,2}(I,\sfd_{eucl},\meas)$ with $\norm{v}_{L^{2}(\mm)}=1$ satisfying 
\begin{equation}\label{eq:smallgap}
\int_I\abs{v'}^2\mm-N<\beta.
\end{equation}
Then it holds
\begin{equation}\label{eq:controldistsob}
 \min\left\lbrace\norm{v-u^*}^2_{W^{1,2}(\mm)},\norm{v+u^*}^2_{W^{1,2}(\mm)} \right\rbrace\le C\left(\int\abs{v'}^2\mm-\int\abs{(u^*)'}^2\mm\right), 
\end{equation}
where $u^{*}$ is a first Neumann eigenfunction with $\|u^{*}\|_{L^{2}(\mm)}=1$.
\end{proposition}

\begin{proof}
We begin by rewriting
\begin{align}
\int\abs{v'}^2\mm-\int\abs{(u^*)'}^2\mm=&\int\abs{(v-u^*)'}^2\mm+2\int (v-u^*)'\,(u^*)' \,\mm \nonumber\\
=& \int\abs{(v-u^*)'}^2\mm- 2  \lambda^{1,2}_{(I,\sfd_{eucl},\mm)} \left( 1- \int v \,u^{*} \,\mm \right)  \nonumber\\
=& \int\abs{(v-u^*)'}^2\mm-   \lambda^{1,2}_{(I,\sfd_{eucl},\mm)}  \int (v-u^{*})^{2} \,\mm. \label{eq:comput}
\end{align}
Now \eqref{eq:smallgap} implies that $\abs{\int vu^*\, \mm}>1/2$ by \autoref{cor:useful}. Hence, assuming without loss of generality that $\int u^*v \, \mm>1/2$, we get $\abs{\int u^*(u^*-v)\mm}<1/2$. Therefore, \autoref{cor:useful} yields
\begin{equation*}
\int\abs{(v-u^*)'}^2\mm\ge (\lambda^{1,2}_{(I,\sfd_{eucl},\mm)}+\beta)\norm{v-u^*}_2^2.
\end{equation*}
The combination of the last estimate with \eqref{eq:comput} gives
\begin{equation}\label{eq:controldist}
\norm{v-u^*}^2_{2} \le C\left(\int\abs{v'}^2\mm-\int\abs{(u^*)'}^2\mm\right),
\end{equation}
with $C:=1/\beta$. We now improve \eqref{eq:controldist}  to $W^{1,2}$-closeness, namely \eqref{eq:controldistsob}. 
  In order to do so, it suffices to observe that the estimates we obtained above yield
\begin{align*}
\int\abs{(v-u^*)'}^2\mm\le &\, \lambda^{1,2}_{(I,\sfd_{eucl},\mm)}\norm{v-u^*}^{2}_{2}+\int\abs{v'}^2\mm-\int\abs{(u^*)'}^2\mm\\
\le &\, C(1+\lambda^{1,2}_{(I,\sfd_{eucl},\mm)})\left(\int\abs{v'}^2\mm-\int\abs{(u^*)'}^2\mm\right).
\end{align*}
\end{proof}

\begin{theorem}[1-dimensional Quantitative Obata's Theorem on the function]\label{thm:mainthm1d}
For every $N>1$ there exist constants $C=C(N)>0$ and $\delta_0= \delta_0(N)>0$ with the following property.
Let $([0,D],\dist_{eucl},\meas)$ be a one dimensional $\CD(N-1,N)$ m.m.s. and let  $u\in\Lip(I)$ satisfy $\int u \,\mm=0$ and $\int u^2 \,\mm=1$.
If
\begin{equation*}
\delta:=\int \abs{u'}^2\mm-N\le\delta_0,
\end{equation*}
then 
\begin{equation}\label{eq:thm1d}
\min\left\lbrace\norm{u-\sqrt{N+1}\cos(\cdot)}_{W^{1,2}(\meas)},\norm{u+\sqrt{N+1}\cos(\cdot)}_{W^{1,2}(\meas)} \right\rbrace\le C\delta^{\min\left\lbrace 1/2,1/N\right\rbrace }. 
\end{equation}
Moreover, setting  $r=\delta^{\gamma/N}$ for some $\gamma\in (0,1)$,  for any $\eta\in (0, r/10)$ it holds
\begin{align}
&\min\left\lbrace\norm{u-\sqrt{N+1}\cos(\cdot)}_{W^{1,2}([0,r],\meas)},\norm{u+\sqrt{N+1}\cos(\cdot)}_{W^{1,2}([0,r],\meas)} \right\rbrace \nonumber\\
& \qquad \qquad \le C\left(\delta^{1/2}+r^{N/2}\delta^{{\min\left\lbrace 1/2,1/N\right\rbrace}}\right).  \label{eq:impr0r2}\\
&\min\left\lbrace \norm{u^*-\sqrt{N+1}\cos(\cdot)}_{W^{1,2}([r-\eta, r+\eta],\mm)},\norm{u^*+\sqrt{N+1}\cos(\cdot)}_{W^{1,2}([r-\eta, r+\eta],\mm)}\right\rbrace \nonumber \\
& \qquad \qquad \le C(N)\left(\delta^{1/2}+(r^{N-1} \eta)^{1/2}\delta^{\min\{1/2,1/N\}} \right). \label{eq:impr0r21}
\end{align}
\end{theorem}

\begin{proof}
First apply \autoref{prop:deficitestimatesdist1} to bound the $W^{1,2}(\meas)$-distance between $u$ and a first eigenfunction of the Neumann Laplacian on $([0,D],\dist_{eucl},\meas)$, then apply \autoref{prop:estdisteigcos} (respectively \autoref{rm:impr0r1}) to bound the $W^{1,2}(\meas)$-distance (respectively the $W^{1,2}([0,r],\meas)$ or $W^{1,2}([r-\eta, r+\eta],\meas)$ distance)  between the first eigenfunction and the normalized cosine. The sought estimate follows by the triangle inequality. 
\end{proof}


\section{Quantitative Obata's Theorem on the diameter}\label{S:Obatadiam}

Building on top of the one-dimensional results obtained in Section \ref{S:1d},
we will derive several quantitative estimates for 
a general essentially non-branching  m.m.s. $(X,\sfd,\mm)$ verifying $\CD(K,N)$.

Given a m.m.s. $(X,\sfd,\mm)$, the perimeter  $\mathsf{P}(E)$ of a Borel subset  $E \subset X$ is defined as 
\begin{equation}  \label{eq:defPer}
\mathsf{P}(E):= \inf\left\{\liminf_{n\to \infty} \int_{X} |\nabla u_n| \,\mm \,:\,  u_n \in \Lip(X), \, u_n\to \chi_E \text{ in } L^1_{\loc}(X)\right\}, 
\end{equation}
where  $\chi_E$ is the characteristic  function of $E$. Accordingly $E \subset X$ has finite perimeter in $(X,\sfd,\mm)$ if and only if $\mathsf{P}(E)< \infty$. 
\\The isoperimetric profile $\cI_{(X,\sfd,\mm)}:[0,1]\to [0,\infty)$ is given by 
\begin{equation}  \label{eq:defIsopProf}
\cI_{(X,\sfd,\mm)}(v):=\inf\{\mathsf{P}(E)\,: \, E \subset X, \, \mm(E)=v\} .
\end{equation}
Given a smooth Riemannian manifold $(M,g)$ with finite Riemannian volume ${\rm vol}_{g}(M)<\infty$, let us denote $\mm_{g}:= \frac{1}{{\rm vol}_{g}(M)} \, {\rm vol}_{g}$ the normalized Riemannian volume measure.
\\ We next recall the improved Levy-Gromov inequality obtained by Berard-Besson-Gallot \cite[Remark 3.1]{BerardBessonGallot85} for smooth  Riemannian $N$-manifolds with Ricci $\geq N-1$ and with upper bound on the diameter (see also \cite{Milman15}).

\begin{theorem}\label{thm:BBG}
Let $(M,\dist,\meas_{g})$ be the metric measure space associated to a Riemannian manifold $(M,g)$ with dimension $N\in \N, N\geq 2$, Ricci bounded from below by $N-1$ and diameter $D$ (recall that, by the Bonnet-Myers Theorem, $D\leq \pi$). Then, for any $v\in(0,1)$, it holds
\begin{equation}\label{eq:improvedLevyGromov}
\frac{\mathcal{I}_{(X,\dist,\meas)}(v)}{\mathcal{I}_{N}(v)}\ge \left(\frac{\int_0^{\frac{\pi}{2}}\left(\cos t\right)^{N-1}\di t}{\int_0^{\frac{D}{2}}\left(\cos t\right)^{N-1}\di t}\right)^{\frac{1}{N}}=:C_{N,D}\ge 1,
\end{equation}
where $\cI_{N}$, defined in    \eqref{eq:exprlower},  for $N\geq 2, N\in \N$ is the isoperimetric profile of the normalized round sphere of constant sectional curvature one $({\mathbb S}^{N}, \dist_{{\mathbb S}^{N}},\meas_{g_{{\mathbb S}^{N}}})$.
\end{theorem}

We  extend \autoref{thm:BBG} to the class of essentially non branching $\CD(N-1,N)$ metric measure spaces, $N>1$ any Real parameter. In view of \cite{CavallettiMondino17a, CavallettiMondino18} the result follows from the 1-dimensional  improved Levy-Gromov Inequality proved in \autoref{thm:improvedLevyGromov1D}.

\begin{theorem}[Berard-Besson-Gallot improved Levy-Gromov for $\CD(N-1,N)$ e.n.b. spaces]\label{thm:improvedLevyGromov}
Let $(X,\dist,\meas)$ be an essentially non branching $\CD(N-1,N)$ m.m.s.  with $\diam(X)\leq D$, for some $N>1,D\in (0,\pi]$. Then, for any $v\in (0,1)$, it holds
\begin{equation}\label{eq:imprlevgrom}
\frac{\mathcal{I}_{(X,\dist,\meas)}(v)}{\mathcal{I}_{N}(v)}\ge  \left(\frac{\int_0^{\frac{\pi}{2}}\left(\cos t\right)^{N-1}\di t}{\int_0^{\frac{D}{2}}\left(\cos t\right)^{N-1}\di t}\right)^{\frac{1}{N}}=:C_{N,D}\ge 1,
\end{equation}
where $\cI_{N}$ was  defined in    \eqref{eq:exprlower}.
\end{theorem}

\begin{proof}
 One of the main results in \cite{CavallettiMondino17a, CavallettiMondino18} is that for $(X,\sfd,\mm)$ as in the assumptions of the theorem it holds
\begin{equation}\label{eq:IXdmIND}
\mathcal{I}_{(X,\dist,\meas)}(v)\ge \mathcal{I}_{N,D}(v),
\end{equation}
where $\mathcal{I}_{N,D}$ stands for the model isoperimetric profile defined in \eqref{eq:exprlower}. 
\\ The claimed \eqref{eq:imprlevgrom} follows by combining \eqref{eq:IXdmIND} with \autoref{thm:improvedLevyGromov1D}.
\end{proof}

It is also possible to obtain a quantitative spectral gap inequality for Neumann boundary conditions.
The analogous result in the case of smooth Riemannian manifolds was established in \cite[Theorem B]{Croke82} building upon a quantitative improvement of the Levy-Gromov inequality and on \cite{BerardMeyer82} (see also \cite[Corollary 17]{BerardBessonGallot85}).

\begin{theorem} [Improved spectral gap and quantitative Obata's Theorem for $\CD(N-1,N)$ e.n.b. spaces]\label{thm:improvedspectralNeumann}
Let $(X,\dist,\meas)$ be an essentially non branching $\CD(N-1,N)$ m.m.s.  with $\diam(X)\leq D$, for some $N>1,D\in (0,\pi]$.
Then 
\begin{equation}\label{eq:firstnonu}
\lambda^{1,2}_{(X,\dist,\meas)}\ge N C_{N,D}^2,
\end{equation}
where $C_{N,D}$ is given in \eqref{eq:imprlevgrom}. Moreover, there exists $C=C_{N}>0$ (more precisely one can choose $C_{N}=\bar{C} N$ where  $\bar{C}$ was defined in  \autoref{prop:asymptoticC}) such that 
\begin{equation*}
C_{N}(\pi-\diam(X))^N\le \lambda^{1,2}_{(X,\sfd,\mm)}-N.
\end{equation*}  
\end{theorem}

\begin{proof}
Thanks to \cite[Theorem 4.4]{CavallettiMondino17b} (see also \autoref{thm:improvedspectralNeumann1D}) we know that $\lambda^{1,2}_{(X,\dist,\meas)}\geq \lambda^{1,2}_{N,D}$, where $\lambda^{1,2}_{N,D}$  was defined in \eqref{eq:eigenvalueproblem}.

Let us briefly outline the argument since it will be relevant for addressing the quantitative inequality for the first eigenfunction later in the note.
By the very definition of $ \lambda^{1,2}_{(X,\sfd,\mm)}$ it suffices to prove that, for any $u\in\Lip(X)$ with $\int u \,\mm=0$ and $\int u^2 \, \mm=1$, it holds
\begin{equation*}
\delta(u):=\int_{X}\abs{\nabla u}^2\mm -N\ge C_{N}(\pi-\diam(X))^N.
\end{equation*}
To this aim, we perform the 1D-localization associated to the function $u$ which by assumption has null mean value  (this is analogous to the proof  of \cite[Theorem 4.4]{CavallettiMondino17b}; see  Section \ref{Ss:localization} for some basics about 1D-localization). We obtain
\begin{align*}
\int_{X}\abs{\nabla u}^2\mm-N\int_{X} u^2\, \mm\ge& \int_{Q}\left(\int_{X_q}\abs{u_q'}^2\mm_q-N\int_{X_q} u_q^2 \, \mm_q\right) \, \qq(dq)\\
\ge&\int_{Q}\left(\lambda^{1,2}_{N,\diam(X_q)}\int_{X_q} u_q^2 \,\mm_q-N\int_{X_q} u_q^2 \, \mm_q\right)\, \qq(dq)\\
\ge&\int_{Q}(\lambda^{1,2}_{N, \diam(X)}-N)\int_{X_q}u_q^2 \, \mm_q\, \qq(dq)=\lambda^{1,2}_{N,\diam(X)}-N.
\end{align*}

Taking into account \autoref{thm:improvedspectralNeumann1D}, we conclude that
\begin{equation*}
\delta(u)\ge\lambda^{1,2}_{\diam(X),N}-N\ge C_{N}(\pi-\diam(X))^{N}
\end{equation*}
and \eqref{eq:firstnonu} can be obtained in an analogous way.
\end{proof}

\begin{remark}
In \cite{JiangZhang16} the authors obtained a quantitative version of the estimate for the gap of the diameters in terms of the deficit in the spectral gap for $\RCD$ spaces  (see Remark 1.3 therein). Their estimate reads as follows: if $(X,\dist,\meas)$ is an $\RCD(N-1,N)$ space of diameter $D\le\pi$, then 
\begin{equation*}
\lambda^{1,2}_{(X,\sfd,\mm)}\ge \frac{N}{1-\cos^N(D/2)}.
\end{equation*}
\autoref{thm:improvedspectralNeumann} extends such quantitative control to essentially non-branching $\CD(N-1,N)$  spaces whose Sobolev space $W^{1,2}$ is a priori non-Hilbert (but just Banach, as for instance on Finsler manifolds).
\end{remark}
\subsection{Volume control}
The aim of this brief subsection is to prove that for a $\CD(N-1,N)$ m.m.s. with diameter close to $\pi$ we have a quantitative volume control for balls centred at extrema of long rays. The proof is inspired by \cite[Lemma 5.1]{Ohta07} where the case of maximal diameter $\pi$ is treated (see also \cite[Proposition 5.1]{CMM}).

\begin{proposition}\label{prop:volumeest}
Let $(X,\dist,\meas)$ be a m.m.s. satisfying $\CD(N-1,N)$ (actually ${\rm MCP}(N-1,N)$ is enough). Let $P_{N},P_{S}\in X$ be such that $\dist(P_{N},P_{S})=\pi-\delta$, for some $\delta\geq 0$. Then, for any $0<r<\pi-\delta$, it holds
\begin{equation}\label{eq:volumeestimate}
\meas_N([0,r])\le\meas(B_{r}(P_{N}))\le \meas_N([0,r])+\meas_N([r,r+\delta]),
\end{equation}
where we recall that $\meas_N=\frac{1}{\omega_N}\left(\sin t\right)^{N-1}\di t$ is the model measure on the interval $[0,\pi]$.

\end{proposition}

\begin{proof}
First of all, since $\dist(P_{N}, P_{S})=\pi-\delta$, it holds $B_{r}(P_{N})\cap B_{\pi-r-\delta}(P_{S})=\emptyset$.\\
Thanks to the Bishop-Gromov inequality implied by the  $\CD(N-1,N)$ condition (actually ${\rm MCP}(N-1,N)$ is enough), and using that $\mm(X)=1$, we have
\begin{equation}
\meas(B_{r}(P_{N}))\ge\meas_N([0,r]),\quad \meas(B_{\pi-r-\delta}(P_{S}))\ge\meas_N([0,\pi-r-\delta])=\meas_N([r+\delta,\pi]),
\end{equation}
where the last equality follows from the symmetries of the density $\sin^{N-1}(\cdot)$. Hence we can compute 
\begin{align*}
\meas(B_{r}(P_{N}))\leq&~ 1- \meas(B_{\pi-r-\delta}(P_{S})) \leq  1- \meas_N([0,\pi-r-\delta])\\
=&~ \meas_N([0,r])+\meas_N([r,r+\delta]).
\end{align*}
The claimed conclusion \eqref{eq:volumeestimate} follows.
\end{proof}

\section{Quantitative Obata's Theorem on almost optimal functions}\label{sec:QuantObataFuct}
Consider $u \in \Lip(X)$ such that 
$$
\int_{X} u \mm = 0, \qquad  \int_{X} u^{2}\mm = 1; 
$$
denote its spectral gap  deficit with 
\begin{equation}\label{eq:defdelta(u)Sec5}
\delta (u) : = \int_{X}|\nabla u|^{2} \mm - N.
\end{equation}
Since we are interested in quantitative estimates when the spectral gap deficit is small, it is enough to consider the case $\delta(u)\leq 1$.
Recall that $N$ is the first eigenvalue for the Neumann Laplacian 
for the 1-dimensional metric measure space $([0,\pi], |\cdot|, \meas_{N})$
where $\meas_{N} : = \sin^{N-1}(t)dt/\omega_{N}$ and $\omega_{N}$ is the normalizing constant. In particular 
$$
N = (N+1)\int_{(0,\pi)} \sin^{2}(t)\mm_{N}(dt),
$$
since, as we already observed, $\int_{(0,\pi)}\cos^2(t)\mm_N(\di t)=1/(N+1)$.

Consider the localization associated to the zero-mean function $u$ (see \autoref{Ss:localization} for the background and for the relevant bibliography): 
$$
\meas\llcorner_{\mathcal{T}} = \int_{Q} \meas_{q} \, \qq(dq),
$$
where $\mathcal{T}$ is the transport set associated to the $L^{1}$-optimal transport problem between $u^{+}\meas$ and $u^{-}\meas$, 
the positive and the negative part of $u$, respectively. 
It follows that 
\begin{equation}\label{eq:startingPoint}
\int_{Q} \int_{X_{q}}|u|^{2} \mm_{q} \, \qq(dq)  = \int_{\mathcal{T}} |u|^{2} \mm = \int_{X} |u|^{2} \mm = 1, \qquad   \int_{X \setminus \mathcal{T}} |\nabla u |^{2} \mm = 0. 
\end{equation}
Setting $u_{q}:=u|_{X_{q}}$ and  $|c_{q}| : = \left(\int_{X_{q}} |u_{q}|^{2} \mm_{q}\right)^{1/2}$ (for the sign of $c_{q}$, see before \eqref{eq:1Duqcosdq}), observe that \eqref{eq:startingPoint} gives
\begin{equation}\label{intcq2=1}
\int_{Q} c_{q}^{2} \,  \qq(dq)=1.
\end{equation}
Moreover, the integral constraint $\int_{X} u\,\mm=0$ localizes to almost every ray:
\begin{equation}\label{eq:intuq=0}
\int_{X_{q}} u_{q} \, \mm_{q}=0.
\end{equation}
Since almost each ray $(X_{q}, \sfd|_{X_{q}}, \mm_{q})$ is a 1-dimensional $\CD(N-1,N)$ space,  the Lichnerowicz spectral gap gives
\begin{equation}\label{eq:1dimspectralgap}
\int_{X_{q}} |u_{q}'|^{2} \, \mm_{q}\geq N c_{q}^{2},
\end{equation}
where  $|u_{q}'|(x)$ denotes the local Lipschitz constant of  $u_{q}: (X_{q}, \sfd|_{X_{q}})\to \R$ at $x\in X_{q}$. It is clear that, for each $x\in X_{q}\subset X$,  $|u_{q}'|(x)$  is bounded by the local Lipschitz constant $|\nabla u|(x)$ of $u: (X, \sfd)\to \R$:
\begin{equation}\label{uq'leqnablau}
|u_{q}'|(x)\leq |\nabla u|(x), \quad \forall x\in X_{q}, \, \qq\text{-a.e. }q\in Q.
\end{equation}
With a slight abuse of notation, in order to keep the formulas short,  in the following we will often identify $\qq$ and $\qq\llcorner_{ \{q\in Q: \,  c_{q}> 0\}}$.
Localizing the spectral gap deficit using \eqref{uq'leqnablau} gives

\begin{align}
\delta(u) = &~\int_{X}|\nabla u|^{2} \mm- N 
\geq ~ \int_{Q} \left( \int_{X_{q}}\frac{|u_{q}'|^{2}} {c_{q}^{2}} \mm_{q}  \right) \,
c_{q}^{2}\,\qq(dq) -N \nonumber \\
= &~
\int_{Q} \left[ \int_{X_{q}} \left( \frac{|u_{q}'|^{2}}{c_{q}^{2}} - N \right)
\mm_{q}  \right]c_{q}^{2} \,\qq(dq)  \\ 
= &~  \int_{Q}\delta(u_{q})c_{q}^{2}\, \qq(dq), \label{eq:deltageqdeltaq}
\end{align}
where we set
$$
\delta(u_{q}) : = \int_{X_{q}} \left( \frac{|u'_{q}|^{2}}{c_{q}^{2}} - N \right)\,\mm_{q},
$$
the one-dimensional spectral gap deficit of $u_{q}$.
From now on, in order to keep notation short, we will write $\delta$ for $\delta(u)$. 
Let $\beta\in (0,1)$ be a real parameter to be optimised later in the proof and denote the set of ``long rays'' by 
$$
Q_{\ell} : = \{ q \in Q \colon    \delta(u_{q}) \leq \delta^{\beta}  \text{ and } c_{q}>0\}.
$$
It follows  from \eqref{eq:deltageqdeltaq}, Chebyshev's inequality and  \eqref{intcq2=1} that 
\begin{equation}\label{eq:measQlong}
\int_{Q \setminus Q_{\ell}} c_{q}^{2}\,\qq(dq)  \leq  \delta^{1-\beta}, \qquad \int_{Q_{\ell}}c_{q}^{2}\,\qq(dq) \geq 1 -  \delta^{1-\beta}.
\end{equation}
Hence we can use \autoref{thm:improvedspectralNeumann1D} 
to deduce that for all $q\in Q_{\ell}$,
\begin{equation}\label{eq:longrays}
(\pi - |X_{q}|)^{N} \leq C_{N} \delta^{\beta},
\end{equation}
where  $|X_{q}|$ denotes the length of the ray $X_{q}$. 
Being the preimage of a measurable function, $Q_{\ell}$ is a  measurable subset of $Q$.
Adopting the  notation $R(E):=\cup_{q\in E} X_{q}$, so that $R(E)$ is the span of the rays corresponding to equivalence classes in $E$, we claim that
\begin{equation}\label{eq:intgradbadset}
\int_{X\setminus R(Q_{\ell})}\abs{\nabla u}^2\mm \leq \   (N+1) \delta^{1-\beta}.
\end{equation}
Indeed \eqref{uq'leqnablau}, \eqref{eq:1dimspectralgap}  and \eqref{eq:measQlong} yield
\begin{equation*}
\int_{R(Q_{\ell})}\abs{\nabla u}^2\mm
\ge\int_{Q_{\ell}}\int_{X_q}\abs{u_q'}^2 \mm_q\, \qq(dq)
\ge N\int_{Q_{\ell}}c_q^2\, \qq(dq) \ge N(1-\delta^{1-\beta}).
\end{equation*}
The claim \eqref{eq:intgradbadset} follows by  combining the last estimate with
\begin{equation*}
\int_{X\setminus R(Q_{\ell})}\abs{\nabla u}^2\mm+\int_{R(Q_{\ell})}\abs{\nabla u}^2\mm=\int_{X}\abs{\nabla u}^2\mm\le N+\delta.
\end{equation*} 
For each $q\in Q$, we denote with $a(X_{q})$ (resp. $b(X_{q})$) the initial (resp. final) point of the  ray $X_{q}$.
\smallskip

Throughout this last section we will often make the identification between the ray $X_q$ and the interval $(0,|X_q|)$.

\begin{proposition}\label{prop:PNPS}
There exists 
a distinguished $\bar q \in Q_{\ell}$ having initial point 
$P_{N}$ and final point $P_{S}$ such that
\begin{equation}\label{eq:neartoNS2}
\dist(P_{N}, a(X_{q}))\leq C(N) \delta^{\beta/N}, \qquad  \dist(P_{S}, b(X_{q})) \leq C(N) \delta^{\beta/N},\quad  \forall q \in Q_{\ell}.
\end{equation}
\end{proposition}

\begin{proof}
Fix any $\bar q \in Q_{\ell}$ and set $P_{N} : = a(X_{\bar q})$, 
$P_{S} : = b(X_{\bar q})$.
By $\dist$-cyclical monotonicity of the transport set ${\mathcal T}$, for any other $q \in Q_{\ell}$ it holds
$$
2\pi - \dist(a(X_{q}),b(X_{q}) ) - \dist(P_{N},P_{S}) \geq 
2\pi - \dist(a(X_{q}), P_{S} ) - \dist(b(X_{q}), P_{N} ), 
$$
which we rewrite as 
$$
\pi - |X_{q}| + \pi - |X_{\bar q}| \geq 
\pi -\dist(a(X_{q}), P_{S} ) + \pi - \dist(b(X_{q}), P_{N} ).
$$
Combining the last estimate with \eqref{eq:longrays} gives
$$
2C_{N} \delta^{\beta/N} \geq \pi -\dist(a(X_{q}), P_{S} ) + \pi - \dist(b(X_{q}), P_{N} ).
$$
Finally by \cite[Proposition 5.1]{CMM} we deduce 
the existence of a constant, depending only on the dimension $N$, such that 
$$
\dist(a(X_{q}), P_{N}) \leq  C(N) \delta^{\beta/N}, \quad  \dist(b(X_{q}), P_{S}) \leq C(N) \delta^{\beta/N},
$$
and the claim follows.
\end{proof}

From now on, for every $q\in Q_{\ell}$ choose the sign of $c_{q}$ so that 
\begin{align*}
\left\| \frac{u_{q}}{c_{q}} \right. & - \left.\sqrt{N+1}\cos(\cdot) \right\|_{L^{2}(X_{q},\mm_{q})}\\
&~ = \min \left\{
\left\| \frac{u_{q}}{|c_{q}|} + \sqrt{N+1} \cos(\cdot) \right\|_{L^{2}(X_{q},\mm_{q})}, \left\| \frac{u_{q}}{|c_{q}|} - \sqrt{N+1}\cos(\cdot) \right\|_{L^{2}(X_{q},\mm_{q})} \right\}.
\end{align*}
From \autoref{thm:mainthm1d} we  obtain that for all $ q\in Q_{\ell}$ it holds:
\begin{equation}\label{eq:1Duqcosdq}
\left\| \frac{u_{q}}{c_{q}} - \sqrt{N+1}\cos(\cdot) \right\|_{L^{2}(X_{q},\mm_{q})} \leq C(N)  \delta^{\beta \min\{1/2, 1/N\}}.
\end{equation}
The goal of the next section is to globalise  estimate \eqref{eq:1Duqcosdq} to the whole space $X$.

The sought bound will be obtained through two intermediate steps: firstly, in \autoref{prop:varcqlongrays}, we control  the variance of the map $q\mapsto c_q$ w.r.t. the measure $\qq$ on the set of long rays $Q_{\ell}$. 
Then, in \autoref{prop:massoflongraysbound}, we estimate $(1-\qq(Q_{\ell}))$ in terms of a power of the deficit.


Below we briefly present the strategy of the proof.  In order to fix the ideas, we discuss the heuristics in the rigid case of zero deficit. Actually in the case of zero deficit there is a more streamlined argument (the assumption that $u$ is Lipschitz, combined with the forth bullet below, gives immediately that $q\mapsto c_{q}$ is constant), however the point here is to present a strategy which generalises to the non-rigid case of non-zero deficit.

In the case where $\delta(u)=0$, the results of the previous sections give the following conclusions:
\begin{itemize}
\item Almost all the transport rays have length $\pi$. Moreover: they start from a common point $P_{N}$, with $u(P_{N})>0$,  and end in a common point $P_{S}$, with $u(P_{S})<0$;
\item $\meas(B_r(P_{N}))=\meas_N([0,r])$, for any $r\in[0,\pi]$;
\item For $\qq$-a.e. $q\in Q$, it holds that $\meas_q=\mm_N$ is the model measure for the $\CD(N-1,N)$ condition;
\item For $\qq$-a.e. $q\in Q$, it holds that $u_q(\cdot)=c_q\cos(\dist(P_{N},\cdot))$.
\end{itemize}
Our aim is to prove that $\qq(Q)=1$ and that $c_q=1$  for $\qq$-a.e. $q\in Q$.  The basic idea is to apply the Poincaré inequality to balls centred at $P_{N}$ and having radii converging to $0$.\\

Observe that we can compute 
\begin{equation}\label{eq:fintBrPNu}
\fint_{B_r(P_{N})}u \,\mm=\frac{1}{\mm_N([0,r])}\int_{Q}\int_0^rc_q\cos(t)\meas_N(dt)=\left(\int_Qc_q \,\qq\right)\fint_0^r\cos(t)\mm_N(dt).
\end{equation} 
Moreover, recalling that $u=0$ $\mm$-a.e. outside of the transport set, we have
\begin{align}
&\fint_{B_r(P_{N})}\abs{u-\fint_{B_r(P_{N})} u \,\mm}^2\meas \nonumber \\
\overset{ \eqref{eq:fintBrPNu}}{=}&\left(1-\qq(Q)\right)\left(\fint_{B_r(P_{N})}u \, \mm\right)^2+\int_{Q}\fint_0^r\abs{c_q\cos(t)-\int_Qc_q \,\qq(dq)\fint_0^{r}\cos(t)\,\mm_N(dt)}^2\,\mm_N(dt) \,\qq(dq) \nonumber  \\
\sim \;&  \left(1-\qq(Q)\right)\left(\int_Q c_q \, \qq(dq)\right)^2+\int_Q\abs{c_q-\int_Qc_q \,\qq(dq)}^2\qq(dq)\qquad\text{as $r\to 0$}, \label{eq:intBru-fint}
\end{align}
where in the last step we relied on the asymptotic $\cos(t)=1+o(t)$ as $t\to 0$.
Eventually we can compute 
\begin{equation*}
\fint_{B_{2r}(P_{N})}\abs{\nabla u}^2\mm=\int_{Q}c^2_q \,\qq( dq)\fint_0^{2r}\sin^2(t)\, \mm_N (dt)=\fint_{0}^{2r}\sin^2(t)\,\mm_N (dt)\sim r^2\qquad\text{as $r\to 0$},
\end{equation*}
where in the last step we relied on the asymptotic $\sin(t)=t+o(t)$ as $t\to 0$.

An application of the Poincaré inequality, in the asymptotic regime $r\downarrow 0$, yields that
\begin{equation}
\int_Q\abs{c_q-\int_Q c_q \, \qq(dq)}^2\qq(dq)=0,
\end{equation}
which implies both the conclusions $\qq(Q)=1$ and $q\mapsto c_q$ constant $\qq$-a.e.. Due to the constraint $\int_Qc_q^2\qq(dq)=1$ and the fact that $u(P_{N})>0$, we also get that $c_q=1$ $\qq$-a.e., as we claimed.  
\\

A second heuristic motivation of the fact that the oscillation of the map $q\mapsto c_q$ is controlled by (a power of) the deficit is that ``the gradient of $u$ is almost aligned along the rays''  in a quantitative  $L^{2}$-sense, suggesting that $u$ ``should not oscillate much in the direction orthogonal to the rays''. Note that in the current framework of $\CD(K,N)$ spaces there is no scalar product and the set $Q$ is far from regular, this is the reason why we cannot directly implement this heuristic strategy. However let us make precise the fact that ``the gradient of $u$ is almost aligned along the rays''  in a quantitative  $L^{2}$-sense, since this will be used in the arguments below.
\begin{align}
0 & \overset{\eqref{uq'leqnablau}}{\leq} \int_{Q}\left( \int_{X_{q}} |\nabla u|^{2}- |u_{q}'|^{2} \mm_{q} \right) \qq(dq) = \int_{X}  |\nabla u|^{2} \mm -  \int_{Q}\left( \int_{X_{q}}  |u_{q}'|^{2} \mm_{q} \right) \qq(dq), \quad \text{use \eqref{eq:defdelta(u)Sec5},\eqref{eq:1dimspectralgap},} \nonumber \\
& \quad  \leq N+\delta -  N \int_{Q} c_{q}^{2}   \qq(dq) \overset{\eqref{intcq2=1}}{=} \delta. \label{eq:nablau2-uq'2delta}
\end{align}

The proofs of \autoref{prop:varcqlongrays} and \autoref{prop:massoflongraysbound} below are based on the idea we just presented, although being quite technical since one has to handle all the various error terms occurring in the non rigid case $\delta(u)>0$.

\subsection{Control on the variance}\label{Ss:variance}


\begin{proposition}\label{prop:varcqlongrays}
The following estimate holds:
\begin{equation}\label{eq:varcqlongrays}
\int_{Q_{\ell}}\left|c_q-\fint_{Q_{\ell}}c_q\, \qq(dq)\right|^2\, \qq(dq) \leq C(N)\left(\delta^{4\gamma/N}+ \delta^{1-\beta-\gamma+(2\gamma/N)} + \delta^{(\beta-\gamma) \min\{2/N, 1\}}\right),
\end{equation}
for any $0<\beta<1$ and for any $0<\gamma<\min\{\beta, 1-\beta\}$.
\end{proposition}

\begin{proof}
In order to bound the variance of $q\mapsto c_q$ on $Q_{\ell}$ we wish to prove that it can be controlled by an integral depending on the variation of the function $u$ on a small ball $B_r(P_N)$.
Next we will appeal on the fact that in the rigid case the $L^2$-norm squared of the gradient of $u$ on $B_r(P_{N})$ is comparable with $r^{N+2}$ and, at least heuristically, this has to be the case also when dealing with almost rigidity.   
Some intermediate steps are devoted to reduce ourselves to the case where the function $u$ coincides with $c_q\cos(\cdot)$ when restricted to any long ray $X_q$.
\\In order to slightly shorten the notation, we will write $C$  in place of $C(N)$ to denote a dimensional constant.

\smallskip
{\bf Step 1.}\\
We will set $r=\delta^{\gamma/N}$ for a suitable $\gamma\in (0, \beta)$.
First of all, notice that the triangle inequality and  \eqref{eq:neartoNS2} yield
\begin{equation}\label{eq:estNSlr}
[0, r - C\delta^{\beta/N}] \subset X_{q} \cap B_{r}(P_{N}) \subset [0, r + C\delta^{\beta/N}],
\end{equation}
for any $q\in Q_{\ell}$,
where we have identified $[0, r \pm C\delta^{\beta/N}]$ with the set 
$$
\{z \in X_{q} \colon \dist(z,a(X_{q})) \leq  r \pm C\delta^{\beta/N} \}.
$$
The  minimality of the mean combined with the inclusion \eqref{eq:estNSlr}  and with the weak local 2-2 Poincar\'e inequality \eqref{eq:LocPoincqq} gives 
\begin{align}\label{E:poincare1}
\int_{Q_{\ell}\times[0,  r - C\delta^{\beta/N}]}\Big|u - 
\fint_{Q_{\ell}\times [0,  r - C\delta^{\beta/N}]} u \, \mm\Big|^{2} 
\mm  \nonumber
\leq & \int_{B_{r}(P_{N})}\Big|u - \fint_{B_{r}(P_{N})} u\Big|^{2} \mm\\
\leq & C r^2\int_{B_{10r}(P_{N})} |\nabla u|^{2} \mm.
\end{align}

{\bf Step 2.} \\
Next we will obtain a more explicit expression of 
$\fint_{Q\times[0,  r - C\delta^{\beta/N}]} u \mm$.\\
Recall that we will often tacitly identify the ray $X_q$ with the interval $(0,|X_q|)$.

Using \autoref{thm:mainthm1d},  \autoref{cor:estdens} and that $\delta_{q} \leq \delta^{\beta}$ 
for $q \in Q_{\ell}$, we estimate 

\begin{align}
\Big| \int_{Q_{\ell}} \int_{[0,r]} &~u \mm_{q}\,\qq(dq)
- \sqrt{N+1}\int_{Q_{\ell}} \int_{[0,r]} c_{q} \cos( \cdot ) \mm_{q}\,\qq(dq) \Big|  \nonumber  \\
\leq  &~
\int_{Q_{\ell}} |c_{q}|
\int_{[0,r]}\left|\frac{u}{c_{q}} - \sqrt{N+1}\cos(\cdot)
\right|\mm_{q}\,\qq(dq)  \nonumber \\
\leq &~ \int_{Q_{\ell}} |c_{q} |
\sqrt{\meas_{q}([0,r])}\left\|\frac{u}{c_{q}} - \sqrt{N+1}\cos(\cdot)
\right\|_{L^{2}([0,r],\meas_{q})}\,\qq(dq)  \nonumber \\
\leq  &~  C \,  r^{N/2} \left( r^{N/2} \delta^{\beta \min\{1/2,1/N\} }+ \delta^{\beta/2}\right) \int_{Q_{\ell}}|c_{q}|\,\qq(dq).  \label{eq:Step2.1}
\end{align}
Also, using  \autoref{cor:estintegrals}, it holds
\begin{align}
\Big| \int_{Q_{\ell}}&~ \int_{[0,r]} c_{q} \cos(\cdot ) \mm_{q}\,\qq(dq)
-  \int_{Q_{\ell}} \int_{[0,r]} c_{q} \cos(\cdot ) \mm_{N}\,\qq(dq) \Big|  \nonumber  \\
\leq &~ C\delta^{\beta/N} r^{N-1} \int_{Q_{\ell}} |c_{q}|\,\qq(dq)
\label{eq:Step2.2}
\end{align}
With an analogous estimate involving  \autoref{cor:estintegrals}, we also obtain
\begin{equation}\label{eq:Step2.3}
|\meas(Q_{\ell} \times [0,r]) - \qq(Q_{\ell})\meas_{N}([0,r])| \leq  C\qq (Q_{\ell}) r^{N-1}\delta^{\beta/N}.
\end{equation}
The combination of \eqref{eq:Step2.1}, \eqref{eq:Step2.2} and  \eqref{eq:Step2.3}, setting $\bar{r}:=r-C\delta^{\beta/N}$, yields
\begin{align}
& \left| \fint_{Q_{\ell}\times[0,\bar{r}]} u \mm 
-\frac{\sqrt{N+1}\int_{Q_{\ell} \times [0,\bar{r}]} c_{q}  \cos(\cdot)  \meas_{N}  \,\qq(dq)}  
{\qq(Q_{\ell})(\meas_{N}([0,\bar{r}]) - C r^{N-1}  \delta^{\beta/N})}\right|  
\nonumber \\
& \qquad  \qquad \leq   \frac{ C (\int_{Q_{\ell}}|c_{q}|\, \qq(dq))( r^{N}  \delta^{\beta \min\{1/2,1/N\} } +r^{N/2} \delta^{\beta/2}+ r^{N-1}  \delta^{\beta/N})} 
{\qq(Q_{\ell})(\meas_{N}([0,\bar{r}]) - C r^{N-1}\delta^{\beta/N})} . \label{eq:Step2.Final}
\end{align}
{\bf Step 3.}\\
In this step we estimate the order in $\delta$ of the right hand side of  \eqref{eq:Step2.Final} and  choose $r$ as
\begin{equation}\label{eq:ChoicerdeltagammaStep3}
 r = \delta^{\gamma/N}, \qquad \text{with} \quad  \gamma \in (0, \beta).
\end{equation}
Approximating the cosine with its first order Taylor expansion near to the origin in \eqref{eq:Step2.Final}, we have
$$
\fint_{Q_{\ell}\times[0,  \bar{r}]} u \mm = 
\frac{\int_{Q_{\ell}} c_{q}\,\qq(dq) +(\int_{Q_{\ell}}|c_{q}|\, \qq(dq)) \,  O ( \delta^{(\beta-\gamma) \min\{1/2, 1/N\}})} {\frac{1}{\sqrt{N+1}}\qq(Q_{\ell})}. $$
Since by Cauchy-Schwartz inequality and  \eqref{intcq2=1} it holds $\left(\fint_{Q_{\ell}}c_q \qq(dq) \right) ^2 \leq \fint_{Q_{\ell}}c_q^{2} \qq(dq) \leq 1/\qq(Q_{\ell}) $, the last estimate can be rewritten as 
\begin{align}\label{eq:estmean}
\left|\fint_{Q_{\ell}\times[0,\bar{r}]} u \mm 
- \sqrt{N+1}\fint_{Q_{\ell}} c_{q}\,\qq(dq)\right|^2
& \leq    \frac{C}{  \qq(Q_{\ell})} \delta^{(\beta-\gamma)\min\{1, 2/N\} }.
\end{align}

\textbf{Step 4.}\\
The aim of this step is to eventually gain \eqref{eq:varcqlongrays}.
We first need the following intermediate inequality, where we assume that $r\gg\delta^{\beta/N}$ is a free parameter, that we will set later:
\begin{align}
\int_{Q_{\ell}} &~ \int_{[0,r]}\Big| u - \sqrt{N+1}c_{q}\Big|^{2} \mm_{q}\,\qq(dq)  \nonumber
\\ \nonumber
&~ \leq 2\int_{Q_{\ell}} \int_{[0,r]}
\Big|u - \sqrt{N+1}c_{q}\cos(\cdot)\Big|^{2}
\mm_{q}\,\qq(dq)  
\\ \nonumber 
&~ \qquad + 
2 \int_{Q_{\ell}} \int_{[0,r]}
\left( \sqrt{N+1}|c_{q}||\cos(\cdot) - 1| \right)^{2}
\mm_{q}\,\qq(dq) 
\\ \nonumber
&~  \leq C \delta^{\beta \min\{1, 2/N\}} r^{N} \int_{Q_{\ell}} c_{q}^{2}
\, \qq(dq)+ C \delta^{\beta} + C r^{4} \int_{Q_{\ell}} c_{q}^{2}  \, \meas_{q}([0,r])\,\qq(dq) ,  \quad \text{from \eqref{eq:impr0r2}}
\\ \nonumber
&\leq  C    \delta^{\beta \min\{1, 2/N\}} r^{N}  + C \delta^{\beta} +
C r^{4} \int_{Q_{\ell}} c_{q}^{2} \, (\meas_{N}([0,r]) +  C r^{N-1} \delta^{\beta/N})\,\qq(dq),\text{ from \eqref{eq:longrays}+\eqref{eq:distancedensitiesImproved} }
\\ 
&~ \leq C  \delta^{\beta \min\{1, 2/N\}} r^{N} +
C r^{4}\meas_{N}([0,r]) \int_{Q_{\ell}} c_{q}^{2} \,\qq(dq)+C \delta^{\beta}\le Cr^{N}( \delta^{\beta \min\{1, 2/N\}}+r^{4})+ C\delta^{\beta}. \label{E:intermediate} 
\end{align}
In particular, the previous inequality holds true 
plugging $\bar{r}:=r - C\delta^{\beta/N}$  in place of $r$, and $r=\delta^{\gamma/N}$ is as in the previous \textbf{Step 3}.
We deduce 
\begin{align}
\nonumber
\meas_{N}&~([0,\bar{r}])\left(N+1\right)\int_{Q_{\ell}}\left|c_q-\fint_{Q_{\ell}}c_q\, \qq(dq)\right|^2\, \qq(dq) \\ \nonumber
\leq &~
\left(N+1\right)\int_{Q_{\ell}}\left|c_q-\fint_{Q_{\ell}}c_q\, \qq(dq)\right|^2(\meas_{q}([0,\bar{r}]) +  C r^{N-1}\delta^{\beta/N})\, \qq(dq) 
\\ \nonumber
\leq &~  C\delta^{\beta/N}r^{N-1}+
\left(N+1\right)\int_{Q_{\ell}}\left|c_q-\fint_{Q_{\ell}}c_q\, \qq(dq)\right|^2\meas_{q}([0,\bar{r}])\, \qq(dq) \\ \nonumber
\leq &~   C\delta^{\beta/N}r^{N-1}+
2\int_{Q_{\ell}}  \int_{[0,\bar{r}]}
\left| u - \sqrt{N+1}c_{q}\right|^{2} \mm_{q}\,\qq(dq)\\ \nonumber
&~ \qquad + 2\int_{Q_{\ell}}
\int_{[0,\bar{r}]}\left|u-\fint_{Q_{\ell}}\sqrt{N+1}c_q\, \qq(dq)\right|^2\mm_{q}\, \qq(dq) \\ \nonumber
\leq &~   C\delta^{\beta/N}r^{N-1}+
2\int_{Q_{\ell}}  \int_{[0,\bar{r}]}\left| u - \sqrt{N+1}c_{q}\right|^{2} \mm_{q}\,\qq(dq)+  
4\int_{Q_{\ell}\times[0,\bar{r}]}\left|u-\fint_{Q_{\ell}\times[0,\bar{r}]}u\mm\right|^2\mm \\ \nonumber
&~ \qquad +
4 \int_{Q_{\ell}}\int_{[0,\bar{r}]}\left|\fint_{Q_{\ell}\times[0,\bar{r}]}u\mm-\sqrt{N+1}\fint_{Q_{\ell}}c_q\, \qq(dq)\right|^2\mm_{q}
\, \qq(dq). 
\end{align}
Now use  \eqref{E:poincare1}, \eqref{eq:estmean}, \eqref{eq:ChoicerdeltagammaStep3}, \eqref{E:intermediate} to continue
the chain of inequalities
\begin{align}
\leq &~  
  C \delta^{\gamma} \left( \delta^{(\beta-\gamma) \min\{1, 2/N\}} +\delta^{4\gamma/N} \right) + C r^2 \int_{B_{10r}(P_{N})} |\nabla u|^{2} \mm.
\label{E:final}
\end{align}
\noindent
Next we wish to bound the term $\int_{B_{10r}(P_{N})}\abs{\nabla u}^2\mm$. To this aim we observe that
\begin{align}
\int_{B_{10r}(P_N)}\abs{\nabla u}^2\mm \le& \int_{X\setminus R(Q_{\ell})}\abs{\nabla u}^2\mm+\int_{Q_{\ell}}\int_0^{10r+C\delta^{\beta/N}}\abs{u'_q}^2\mm_q\qq(d\qq)+\delta,  \quad \text{from \eqref{eq:nablau2-uq'2delta} } \nonumber \\
\le& C(\delta^{1-\beta}+ \delta^{\beta}+r^{N}\delta^{\beta \min\{2/N, 1\}}+\delta^{\beta/N}r^{N-1}r^2)  \nonumber \\
& \quad + C \int_0^{10r+C\delta^{\beta/N}}\sin(\cdot)^2\mm_N, \quad \text{from \eqref{eq:intgradbadset}, \eqref{eq:impr0r2}, \eqref{eq:distancedensitiesImproved}} \nonumber \\
\le & C\left(\delta^{1-\beta}+ \delta^{\beta}+r^{N}\left( \delta^{\beta \min\{2/N, 1\}} +r^2\right)\right).\label{eq:boundgradball}
\end{align}
Combine now \eqref{E:final}  with \eqref{eq:boundgradball} and recall that $r=\delta^{\gamma/N}$, for $0<\gamma<\min\{\beta,1-\beta\}$ to get 
\begin{align*}
\delta^{\gamma}&\int_{Q_{\ell}}\left|c_q-\fint_{Q_{\ell}}c_q\, \qq(dq)\right|^2\, \qq(dq)  \leq  C \delta^{\gamma}\left(\delta^{4\gamma/N}+ \delta^{1-\beta-\gamma+(2\gamma/N)} + \delta^{(\beta-\gamma) \min\{ 1, 2/N\}}\right)
\end{align*}
which gives the desired estimate \eqref{eq:varcqlongrays}. \end{proof}

\subsection{Control of the measure of long rays}

Following \autoref{prop:varcqlongrays},
we set  
\begin{equation}\label{E:mean}
\bar c:=\fint_{Q_{\ell}}c_q\, \qq(dq).
\end{equation}
Next we proceed proving that $\qq(Q_{\ell})$ is quantitatively close to $1$ up to an error of the order of a suitable power of the deficit. 

\begin{proposition}\label{prop:massoflongraysbound}
The following estimate holds:
\begin{equation}\label{eq:boundalpha}
(1-\qq(Q_{\ell}))^2\leq  C(N) \left(\delta^{4\gamma/N}+\delta^{(\beta-\gamma)/N}+\delta^{1-\beta-\gamma}  \right),
\end{equation}
for any $0<\beta<1$ and for any $0<\gamma<\min\{\beta,1-\beta\}$.  
\end{proposition}
\begin{proof}
In order to slightly shorten the notation, we will write $C$ in place of $C(N)$ to denote constants depending only on $N$. Moreover, we will continue to tacitly identify the ray $X_q$ with the interval $(0,|X_q|)$.
We achieve \eqref{eq:boundalpha} through three intermediate steps.

\textbf{Step 1.} \\
Aim of this first step is to prove that, for $r=\delta^{\gamma/N}$, $\gamma \in (0,\min\{\beta,1-\beta\})$, 
letting $\bar{r}:=r-C\delta^{\beta/N}$, it holds
\begin{align}\label{E:step1}
\nonumber
\left(N+1\right)\int_{Q_{\ell}}&\int_{[0,\bar{r}]}\left|c_q\cos(\cdot)-\bar c \,\qq(Q_{\ell})\fint_{[0,\bar{r}]}\cos(\cdot)\mm_N\right|^2\mm_N\, \qq(dq)\\ 
\le& \int_{B_{r}(P_N)}\left|u-\fint_{B_{r}(P_N)} u\mm\right|^2\mm+ C\left(\delta^{\gamma+(\beta-\gamma)/N}+\delta^{1-\beta}\right).
\end{align}
Arguing as in the first steps of the proof of \autoref{prop:varcqlongrays}, we estimate
\begin{align}
\int_{Q_{\ell}}&\int_0^{\bar{r}}\left|\sqrt{N+1}c_q\cos(\cdot)-\fint_{B_{r}(P_N)}u\mm\right|^2\mm_N\, \qq(dq) \nonumber \\
\le  &~\int_{Q_{\ell}} \int_{0}^{\bar{r}} \left|\sqrt{N+1}c_q\cos(\cdot)-\fint_{B_{r}(P_N)}u\mm\right|^2\mm_q\, \qq(dq)+ C \delta^{\beta/N}r^{N-1} 
\;  \text{ from \eqref{eq:distancedensitiesImproved}, \eqref{eq:longrays}} \nonumber \\
\le&~  2\int_{Q_{\ell}}\int_{0}^{\bar{r}}
\Big|\sqrt{N+1}c_q\cos(\cdot)-u\Big|^2\mm_q\, \qq(dq)+  C \delta^{\beta/N}r^{N-1} + 2\int_{Q_{\ell}}\int_{0}^{\bar{r}}\Big|u-\fint_{B_{r}(P_N)}u\mm\Big|^2\mm_q\, \qq(dq)  \nonumber\\
\le &~ 2\int_{Q_{\ell}} c_{q}^{2} \;   \Big\|\frac{u}{c_{q}} - \sqrt{N+1}\cos(\cdot) \Big\|_{L^{2}([0,\bar{r}],\meas_{q})}^{2}\,\qq(dq) +  C \delta^{\beta/N}r^{N-1}  \nonumber\\
&~ + 2\int_{Q_{\ell}}\int_{0}^{\bar{r}}\Big|u-\fint_{B_{r}(P_N)}u\mm\Big|^2\mm_q\, \qq(dq) \mm  \nonumber \\
\le &~    2\int_{B_{r}(P_N)\cap R(Q_{\ell})}\Big|u-\fint_{B_{r}(P_N)}u\mm\Big|^2\mm +  C \delta^{\beta/N}r^{N-1}  \text{ from \eqref{eq:impr0r2}, \eqref{eq:estNSlr}}. \label{eq:Prop2Step1.1}
\end{align}
In order to achieve \eqref{E:step1}, having in mind to argue by triangle inequality, we are left to bound 
\begin{equation}\label{eq:estmeans}
\meas_N([0,\bar{r}])\qq(Q_{\ell})\left|\fint_{B_{r}(P_N)}u\mm-\sqrt{N+1}\bar c\qq(Q_{\ell})\fint_{[0, \bar{r}]}\cos(\cdot)\mm_N\right|^2.
\end{equation}

\noindent
We start by observing that
\begin{align}\label{eq:Prop2Step1.2}
\left|\int _{B_{r}(P_N)}u\mm\right.-&~\left.\sqrt{N+1}\bar{c}\qq(Q_{\ell})\int_0^{r}\cos(\cdot)\mm_N\right|  \nonumber \\
&~ \le
\left|\int_{B_{r}(P_N)\cap R(Q_{\ell})}u\mm-\sqrt{N+1}\bar{c}\qq(Q_{\ell})\int_0^r\cos(\cdot)\mm_N\right|+\left|\int_{B_{r}(P_N)\setminus R(Q_{\ell})}u\mm\right|.
\end{align}
We first treat the second term of the right hand-side.
\\From \eqref{eq:measQlong} we know that $\int_{X\setminus R(Q_{\ell})}u^2\mm\le \delta^{1-\beta}$; an application of H\"older's inequality and \eqref{eq:distancedensitiesImproved} yields 
\begin{equation}\label{eq:secdterm}
\int_{B_{r}(P_N)\setminus R(Q_{\ell})}\abs{u}\mm\le \delta^{(1-\beta)/2} \sqrt{\meas(B_{r}(P_N) \setminus R(Q_{\ell}))}\le C\delta^{(1-\beta)/2}r^{N/2}.
\end{equation}
We estimate the first term in the right hand side of \eqref{eq:Prop2Step1.2} by reducing to  \eqref{eq:Step2.1} in the second step of the proof of \autoref{prop:varcqlongrays}:
\begin{align}
\Big|\int_{B_{r}(P_N)\cap R(Q_{\ell})}u\mm& -\sqrt{N+1}\bar{c}\qq(Q_{\ell})\int_0^r\cos(\cdot)\mm_N\Big| \leq    \Big|\int_{B_{r}(P_N)\cap R(Q_{\ell})}u\mm -  \int_{Q_{\ell}} \int_{[0,r]} u\mm_{q} \, \qq(dq) \Big| \nonumber\\
& \quad+ \Big| \int_{Q_{\ell}} \int_{[0,r]} ~u \mm_{q}\,\qq(dq)
- \sqrt{N+1}\int_{Q_{\ell}} \int_{[0,r]}c_{q} \cos(\cdot) \mm_{q}\,\qq(dq) \Big| \nonumber\\
&  \quad + \Big|   \int_{Q_{\ell}} \int_{[0,r]} \sqrt{N+1}c_{q} \cos(\cdot) \mm_{q}\,\qq(dq)  -\sqrt{N+1}\bar{c}\qq(Q_{\ell})\int_0^r\cos(\cdot)\mm_N \Big|\, .\nonumber
\end{align}
Using  \eqref{eq:distancedensitiesImproved}, \eqref{eq:impr0r2}, \eqref{eq:longrays}, \eqref{eq:estNSlr},    \eqref{eq:Step2.1}, we continue as follows:  
\begin{equation}\label{eq:aa}
\leq \int_{Q_{\ell}} \int_{r-C\delta^{\beta/N}}^{r+C\delta^{\beta/N}} |u| \mm_{q} \, \qq(dq)   
+ C r^{N/2} \left( \delta^{\beta \min\{1/2, 1/N \}}r^{N/2}+ \delta^{\beta/2} +  r^{(N/2)-1} \delta^{\beta/N}\right) \int_{Q_{\ell}}|c_{q}|\,\qq(dq).
\end{equation}
Arguing by triangle inequality bounding first the distance from the normalized cosine (with \eqref{eq:impr0r21}) and then replacing the measures $\mm_q$ with the model measure $\mm_N$ (with \eqref{eq:distancedensitiesImproved}),  we estimate the first summand in the right hand side of \eqref{eq:aa} as  
\begin{equation}\label{eq:abcd}
\int_{Q_{\ell}}\int_{r-C\delta^{\beta/N}}^{r+C\delta^{\beta/N}} |u| \mm_{q} \, \qq(dq) \leq C (r^{N-1}   \delta^{\beta/N}+ r^{(N-1)/2} \delta^{\beta(1/2+1/(2N))} ) \int_{Q_{\ell}}\abs{c_q} \,\qq(dq)  \,.
\end{equation}
Combining  \eqref{eq:aa},  \eqref{eq:abcd}, and choosing $r= \delta^{\gamma/N}$ with $\gamma \in (0,\min\{\beta,1-\beta\})$ yields
\begin{equation}\label{eq:firstterm}
\left|\int_{B_{r}(P_N)\cap R(Q_{\ell})}u\mm -\bar{c}\sqrt{N+1}\qq(Q_{\ell})\int_0^{\bar{r}}\cos(\cdot)\mm_N\right|
\leq C\left(r^{N-1}\delta^{\beta/N}+r^{N}\delta^{\beta \min\{1/2, 1/N \}}+r^{N/2} \delta^{\beta/2}\right).
 \end{equation}
The combination of  \eqref{eq:Prop2Step1.2} \eqref{eq:secdterm} and  \eqref{eq:firstterm} gives
\begin{align}
&\left|\int _{B_{r}(P_N)}u\mm -\bar{c}\sqrt{N+1}\qq(Q_{\ell})\int_0^{\bar{r}}\cos(\cdot) \mm_N\right| \nonumber\\
& \qquad \qquad \le C\left(r^{N-1}\delta^{\beta/N}+r^{N} \delta^{\beta \min\{1/2, 1/N \}}+r^{N/2} \delta^{\beta/2}
 + \delta^{(1-\beta)/2}r^{N/2}\right). \label{eq:Prop2Step1.3}
\end{align}
To bound \eqref{eq:estmeans}, approximating the measure of the ball $B_r(P_N)$ and then the function $u$ with the respective model behaviours, we now estimate    
\begin{align}
\meas_N(&[0,\bar{r}])\qq(Q_{\ell})\left|\fint_{B_{r}(P_N)}u\mm-\bar c\sqrt{N+1}\qq(Q_{\ell})\fint_0^{\bar{r}}\cos(\cdot) \mm_N\right|^2  \nonumber \\
\le&~ 2\meas_N([0,\bar{r}])\qq(Q_{\ell})\left|\fint_{B_{r}(P_N)}u\mm-\frac{1}{\meas_N([0,\bar{r}])}\int_{B_{r}(P_N)}u\mm\right|^2  \nonumber \\
&+2\meas_N([0,\bar{r}])\qq(Q_{\ell})\left|\frac{1}{\meas_N([0,\bar{r}])}\int_{B_{r}(P_N)}u\mm-\bar{c}\qq(Q_{\ell})\sqrt{N+1}\fint_0^{\bar{r}}\cos(\cdot) \mm_N\right|^2  \nonumber
\\
\le &~ 2\meas_N([0,\bar{r}])\qq(Q_{\ell})\left(\int_{B_{r}(P_N)}u\mm\right)^2\left(\frac{1}{\meas(B_{r}(P_N))}-\frac{1}{\meas_N([0,\bar{r}])}\right)^2  \nonumber
\\
&+2\frac{1}{\meas_N([0,\bar{r}])}\qq(Q_{\ell})\left|\int _{B_{r}(P_N)}u\mm-\bar{c}\qq(Q_{\ell})\sqrt{N+1}\int_0^{\bar{r}}\cos(\cdot) \mm_N\right|^2.
\end{align}
Estimate the first term by Cauchy-Schwartz and the second term by  \eqref{eq:Prop2Step1.3}: 
\begin{align}
\le &~ 2\qq(Q_{\ell})
\left[\frac{(\meas(B_{r}(P_N))-\meas_N([0,\bar{r}]))^{2}}{\meas(B_{r}(P_N))\meas_N([0,\bar{r}])} \right]\int_{B_r(P_N)}u^2\mm  \nonumber\\
& \quad  +  C \qq(Q_{\ell})  \left[   \delta^{\gamma+2(\beta-\gamma)/N}+ \delta^{\gamma+ \beta  \min\{1, 2/N\}}+  \delta^{\beta}+\delta^{1-\beta} \right].   \nonumber
\end{align}
Now use \autoref{prop:volumeest}  
and choose $r=\delta^{\gamma/N}$, $\gamma\in(0,\min\{\beta,1-\beta\})$: 
\begin{align}
\le &~ 2\left(\int_{B_r(P_N)}u^2\mm\right)
\left(\frac{\meas_N([\bar{r}, r + C\delta^{\beta/N}])}{\meas_N([0,\bar{r}])}\right)^{2}+C\left(  \delta^{\gamma+2(\beta-\gamma)/N}+\delta^{\gamma+ \beta  \min\{1, 2/N\}}+ \delta^{\beta}+\delta^{1-\beta}  \right) \nonumber\\
\le &~ C\left(   \delta^{\gamma+2(\beta-\gamma)/N}+\delta^{\gamma+ \beta  \min\{1, 2/N\}}+ \delta^{\beta}+  \delta^{1-\beta} \right),\label{eq:Prop2Step1.4}
\end{align}
where the second inequality is obtained by observing that
\begin{align}
\int_{B_r(P_N)}u^2\mm &=  \int_{B_r(P_N)\setminus Q_{\ell}}u^2\mm +  \int_{B_r(P_N)\cap Q_{\ell}}u^2\mm  \nonumber\\
& \leq \delta^{1-\beta}+ 2 \int_{Q_{\ell}}  \int_{[0, r+C\delta^{\beta/N}]} \left( u- c_{q} \sqrt{N+1} \cos(\cdot) \right)^{2} \mm_{q} \qq(dq)  \nonumber\\
& \qquad + 2  \int_{Q_{\ell}}  \int_{[0, r+C\delta^{\beta/N}]} c^{2}_{q} (N+1) \cos^{2}(\cdot) \ \mm_{q} \qq(dq)   \nonumber\\
&\leq C (\delta^{1-\beta}+\delta^{\beta} +\delta^{\gamma+\beta   \min\{1, 2/N\}}+\delta^{\gamma} ).  \nonumber
\end{align}
The claimed estimate \eqref{E:step1} is eventually obtained via triangle inequality from \eqref{eq:Prop2Step1.1}  and  \eqref{eq:Prop2Step1.4}

\smallskip
\textbf{Step 2.} \\
In this second step of the proof, building upon \autoref{prop:varcqlongrays}, we shall obtain the bound
\begin{align}
\int_{Q_{\ell}}&~\int_{[0,\bar{r}]}\left(N+1\right)\left|\bar c\cos(\cdot) -\bar c\,\qq(Q_{\ell})\fint_0^{\bar{r}}\cos(\cdot) \mm_N\right|^2\mm_N\, \qq(dq) \nonumber\\
\leq &~ 2  \int_{B_{r}(P_N)}\left|u-\fint_{B_{r}(P_N)} u\mm\right|^2\mm+C \delta^{\gamma}\left(\delta^{4\gamma/N}+\delta^{(\beta-\gamma)/N} \right)+C \delta^{1-\beta}. \label{eq:Prop2Step2}
\end{align}
Thanks to the triangle inequality, the error we commit replacing $c_q\cos(\cdot)$ with 
$\bar c\cos(\cdot) $ can be controlled by
\begin{align}
\int _{Q_{\ell}}\int_{[0,\bar{r}]} &\left|c_q-\bar c\right|^2\cos^2(t)\mm_N(dt)\, \qq(dq) \leq 
 \meas_N([0,\bar{r}]) \int_{Q_{\ell}}\left|c_q-\bar c\right|^2\, \qq(dq) \nonumber \\
&~ \qquad \le  C \delta^{\gamma}\left(\delta^{4\gamma/N} +\delta^{(\beta-\gamma) \min\{ 1, 2/N\}}\right)+C\delta^{1-\beta+2\gamma/N}, \label{eq:prop2step2.1}
\end{align}
where the last inequality is a consequence of \eqref{eq:varcqlongrays}  and the fact that $\bar{r}\leq r=\delta^{\gamma/N}$, $\gamma \in (0,\min\{\beta,1-\beta\})$.
The claimed \eqref{eq:Prop2Step2}  follows  from   \eqref{eq:prop2step2.1} and \eqref{E:step1}  via triangle inequality. 

\smallskip
\textbf{Step 3.}\\
Using the Taylor expansion $\cos(t)=1+O(t^{2})$ in the left hand side of \eqref{eq:Prop2Step2}, we obtain
\begin{align*}
\int_{Q_{\ell}}& \int_0^{\bar{r}}\left(N+1\right)\Big|\bar c-\bar c\qq(Q_{\ell})\Big|^2\mm_N\, \qq(dq) 
\\
\le &~2  \int_{B_{r}(P_N)}\Big|u-\fint_{B_{r}(P_N)} u\mm\Big|^2\mm+C \delta^{\gamma}\left(\delta^{4\gamma/N}+\delta^{(\beta-\gamma)/N}\right)+C \delta^{1-\beta},
\end{align*}
giving
\begin{align*}
\meas_N([0,\bar{r}])  (N+1)\bar {c}^2(1-\qq(Q_{\ell}))^2
\qq(Q_{\ell})&\\  \le2  \int_{B_{r}(P_N)}\Big|u-\fint_{B_{r}(P_N)} u\mm\Big|^2\mm & +C \delta^{\gamma}\left(\delta^{4\gamma/N}+\delta^{(\beta-\gamma)/N}\right)+C \delta^{1-\beta}.
\end{align*}
Using the 2-2 Poincar\'e inequality \eqref{eq:LocPoincqq} (combined with Bishop-Gromov volume comparison), we obtain
\begin{align}
\meas_N&([0,\bar{r}])\left(N+1\right)\bar{c}^2 (1-\qq(Q_{\ell}))^2 \qq(Q_{\ell})  \nonumber 
\\
\le &~ C r^{2} \int_{B_{10r}(P_{N})} |\nabla u|^{2}\mm +C \delta^{\gamma}\left(\delta^{4\gamma/N}+\delta^{(\beta-\gamma)/N} \right)+C \delta^{1-\beta} 
   \nonumber 
\\
 \leq & ~ C \delta^{\gamma}\left(\delta^{4\gamma/N}+\delta^{(\beta-\gamma)/N} \right)+C \delta^{1-\beta}, \label{eq:Prop2Step3.2}
\end{align}
where in the last estimate we used  \eqref{eq:boundgradball} (recall that $r=\delta^{\gamma/N}$). 
\\Using again that $\int _{Q_{\ell}} \left|c_q-\bar c\right|^2\, \qq(dq)\leq  C\delta^{\alpha(N)}$ from  
\eqref{eq:varcqlongrays} for some $\alpha(N)>0$, observing that 
\begin{equation}\label{eq:intbarc2cq2=intbarccq2}
 \int_{Q_{\ell}}  (c^{2}_{q}-\bar{c}^{2}) \qq(dq)= \int_{Q_{\ell}}\abs{c_q-\bar{c}}^2\qq(dq),
\end{equation} 
 and recalling \eqref{eq:measQlong}, we get:
\begin{align}
\bar{c}^{2} \qq(Q_{\ell})
&= \int_{Q_{\ell}} c^{2}_{q} \qq(dq) +   \int_{Q_{\ell}}  (\bar{c}^{2}-c^{2}_{q}) \qq(dq)\geq 1-\delta^{1-\beta} - \int_{Q_{\ell}}\abs{c_q-\bar{c}}^2\qq(dq) \nonumber \\
& \geq 1-\delta^{1-\beta}-  C\delta^{\alpha(N)}>\frac{1}{C(N)}>0. \label{eq:estbarcqqQl}
\end{align}
Plugging \eqref{eq:estbarcqqQl} into \eqref{eq:Prop2Step3.2} yields:
\begin{equation}\label{eq:Prop2.final}
(1-\qq(Q_{\ell}))^2\leq  C \left(\delta^{4\gamma/N}+\delta^{(\beta-\gamma)/N}+\delta^{1-\beta-\gamma}  \right).
\end{equation}
\end{proof}

\begin{remark}\label{rm:boundmass}
Observe that a direct consequence of \autoref{prop:massoflongraysbound} above is an estimate of the measure of the region of the space which is not covered by transport rays, that is $\{u=0\}$.\\
Indeed \eqref{eq:boundalpha} implies in particular that
\begin{equation}
\mm(X\setminus\mathcal{T})\le 1-\qq(Q_{\ell})\leq C(N) \left(\delta^{2\gamma/N}+\delta^{(\beta-\gamma)/2N} +\delta^{(1-\beta-\gamma)/2}  \right).
\end{equation}
\end{remark}

\subsection{Proof of the main theorem}

We are now ready to prove the main result 
putting together the estimates we proved so far.
First reducing  to the set spanned by long rays using \autoref{prop:massoflongraysbound}; then, building upon \autoref{prop:varcqlongrays} and on \autoref{thm:mainthm1d}, we prove that on the set of long rays the function is close to a fixed multiple of the cosine of the distance from the endpoint. Eventually we change the distance from the endpoint of the ray into the distance from a pole thanks to \eqref{eq:neartoNS2}.

\begin{theorem}\label{thm:quantobata}
For any $N\in (1,\infty)$ there exist $C(N)>0$ and $\delta_0=\delta_{0}(N)>0$ with the following properties.
Let $(X,\dist,\meas)$ be an essentially non branching $\CD(N-1,N)$ m.m.s.. Then, for any $u\in\Lip(X)$ with $\int_Xu\mm=0$, $\int_Xu^2\mm=1$ and 
\begin{equation}
\delta:=\int_X\abs{\nabla u}^2\mm-N\le\delta_0,
\end{equation}
there exists a distinguished point $P\in X$ such that 
\begin{equation}\label{eq:claimMT}
\norm{u-\sqrt{N+1}\cos(\dist(P,\cdot))}_{L^2(\meas)}
\le  C(N)\delta^{1/(6N+4)}.
\end{equation}
\end{theorem}

\begin{proof}
{\bf Step 1.} \\
Let us begin observing that \autoref{prop:varcqlongrays} combined with \eqref{E:mean}  and \eqref{eq:intbarc2cq2=intbarccq2} gives
\begin{equation}\label{eq:MT.Step1.1}
\abs{\int_{Q_{\ell}}c_q^2 \,\qq(dq)-{\bar c}^2\qq(Q_{\ell})} \leq  C(N) \left(\delta^{4\gamma/N}+ \delta^{1-\beta-\gamma+(2\gamma/N)}+ \delta^{(\beta-\gamma) \min\{2/N, 1\}}\right).
\end{equation}
Since from \eqref{eq:measQlong} we know that 
$$
1 -  \delta^{1-\beta}\leq \int_{Q_{\ell}}c_{q}^{2}\,\qq(dq)\leq 1\,,$$ 
and  in \autoref{prop:massoflongraysbound} we proved that
\begin{equation}\label{eq:qQellgeq}
\qq(Q_{\ell})\geq 1-  C(N) \left(\delta^{2\gamma/N}+\delta^{(\beta-\gamma)/2N}+\delta^{(1-\beta-\gamma)/2}  \right),
\end{equation}
from \eqref{eq:MT.Step1.1} we infer that
\begin{equation}\label{eq:1-barc2}
\abs{1-{\bar c}^2}
\leq C(N) \left(\delta^{2\gamma/N}+ \delta^{(1-\beta-\gamma)/2}+ \delta^{(\beta-\gamma)/2N} \right).
\end{equation}
Notice that \eqref{eq:1-barc2} implies (see for instance the proof of \eqref{eq:estbetadefImproved})
\begin{equation}\label{eq:|1-barc|} 
\min\{|1-{\bar c}|, |1+\bar c|\}\le  C(N)  \left(\delta^{2\gamma/N}+ \delta^{(1-\beta-\gamma)/2}+ \delta^{(\beta-\gamma)/2N}\right).
\end{equation}
Without loss of generality (up to switching the sign of $u$) we can assume that  
$$|1-{\bar c}|=\min\{|1-{\bar c}|, |1+\bar c|\}.$$
The combination of  \autoref{prop:varcqlongrays} and  \eqref{eq:|1-barc|} gives
 \begin{align}
\int_{Q_{\ell}}\abs{c_q-1}^2\, \qq(dq)&\le 2 \int_{Q_{\ell}} |c_{q}-\bar{c}|^{2} \qq(dq) + 2 |\bar{c}-1|^{2}  \qq(Q_{\ell}) \nonumber\\
&\leq C(N) \left(\delta^{4\gamma/N}+ \delta^{1-\beta-\gamma}+ \delta^{(\beta-\gamma)/N} \right). \label{eq:varcq1}
\end{align}

{\bf Step 2.}\\
Next we let $P$ be equal to  $P_{N}$  given in  \autoref{prop:PNPS}.
We get: 
\begin{align*}
\big\| u &-\sqrt{N+1}\cos(\dist(P,\cdot))\big\|^2_{L^2(\meas)}
\\
= &~
\int_{Q}\int_{X_q}\abs{u-\sqrt{N+1}\cos(\dist(P,\cdot))}^2\mm_q\, \qq(dq) +\int_{X\setminus\mathcal{T}}(N+1)\cos(\dist(P,\cdot))^2\mm
\\
\le &~\int_{Q_{\ell}}\int_{X_q}\abs{u-\sqrt{N+1}\cos(\dist(P,\cdot))}^2\mm_q\, \qq(dq) \\
&~+2\int_{X\setminus R(Q_{\ell})}u^2\mm+2(N+1)\qq(Q\setminus Q_{\ell}) +(N+1)\mm(X\setminus \mathcal{T}).
\end{align*}
Use \eqref{eq:measQlong}, \eqref{eq:qQellgeq} and  \autoref{rm:boundmass}, recalling that we are tacitly identifying the ray $X_q$ with the interval $(0,|X_q|)$:
\begin{align*}
\le&~ \int_{Q_{\ell}}\int_{X_q}\abs{u-\sqrt{N+1}\cos(\dist(P,\cdot))}^2\mm_q\, \qq(dq) + C(N)  \left(\delta^{2\gamma/N}+\delta^{(\beta-\gamma)/2N} +\delta^{(1-\beta-\gamma)/2}  \right)   \\
\le &~2\int_{Q_{\ell}}\int_{X_q}\abs{u-\sqrt{N+1}\cos(\cdot)}^2\mm_q\, \qq(dq) \\&~+2\int_{Q_{\ell}}\int_{X_q}\abs{\cos(\cdot)-\cos(\dist(P,\cdot))}^2\mm_q\, \qq(dq)+  C(N) \left(\delta^{2\gamma/N}+\delta^{(\beta-\gamma)/2N} +\delta^{(1-\beta-\gamma)/2}  \right).
\end{align*}
Use triangle inequality to estimate the first term and \eqref{eq:neartoNS2} for the second:
\begin{align*}
\le&~ 4\int_{Q_{\ell}}\int_{X_q}\abs{u-\sqrt{N+1}c_q\cos(\cdot)}^2\mm_q\, \qq(dq)+C(N)\int_{Q_{\ell}}\abs{c_q-1}^2\, \qq(dq)\\
&~+ C(N) \left(\delta^{2\gamma/N}+\delta^{(\beta-\gamma)/2N} +\delta^{(1-\beta-\gamma)/2}  \right).
\end{align*}
Use  \eqref{eq:1Duqcosdq} and  \eqref{eq:varcq1}:
\begin{align}
&\le C(N)  \delta^{\beta \min\{1, 2/N\}} \int_{Q_{\ell}}c_q^2\, \qq(dq) +  C(N) \left(\delta^{2\gamma/N}+\delta^{(\beta-\gamma)/2N} +\delta^{(1-\beta-\gamma)/2}  \right)  \nonumber \\
&\le  C(N)  \left(\delta^{2\gamma/N}+\delta^{(\beta-\gamma)/2N} +\delta^{(1-\beta-\gamma)/2}  \right). \label{eq:prefinal}
\end{align}
The optimal choice of parameters in \eqref{eq:prefinal}  is  $\beta=\frac{5N}{6N+4}$ and $\gamma=\frac{N}{6N+4}$ giving the claim \eqref{eq:claimMT}.
\end{proof}


\end{document}